\documentclass[sn-mathphys,Numbered]{sn-jnl}

\usepackage{graphicx}%
\usepackage{multirow}%
\usepackage{amsmath,amssymb,amsfonts}%
\usepackage{amsthm}%
\usepackage{mathrsfs}%
\usepackage[title]{appendix}%
\usepackage{xcolor}%
\usepackage{textcomp}%
\usepackage{manyfoot}%
\usepackage{booktabs}%
\usepackage{algorithm}%
\usepackage{algorithmicx}%
\usepackage{algpseudocode}%
\usepackage{listings}%

\theoremstyle{thmstyleone}%
\newtheorem{theorem}{Theorem}
\newtheorem{corollary}{Corollary}
\newtheorem{lemma}{Lemma}
\newtheorem{proposition}{Proposition}%

\theoremstyle{thmstyletwo}%
\newtheorem{example}{Example}%
\newtheorem{remark}{Remark}%

\theoremstyle{thmstylethree}%
\newtheorem{definition}{Definition}%

\begin{document}
	
	\title[Discrete Quaternionic  Gabor systems]{Discrete Quaternionic (Multi-window) Gabor Systems}
	
	
	\author*[]{\fnm{Najib} \sur{Khachiaa}}\email{khachiaa.najib@uit.ac.ma}
	\affil[]{\orgdiv{Laboratory Partial Differential Equations, Spectral Algebra and Geometry, Department of Mathematics}, \orgname{Faculty of Sciences, University Ibn Tofail}, \orgaddress{\city{Kenitra}, \country{Morocco}}}
	
		\abstract{The aim of this work is to study (Multi-window) Gabor systems in the space \(\ell^2(\mathbb{Z} \times \mathbb{Z}, \mathbb{H})\), denoted by $\mathcal{G}(g,L,M,N)$, and defined by:
\[
\left\{ (k_1,k_2)\in \mathbb{Z}^2\mapsto e^{2\pi i \frac{m_1}{M}k_1} g_l(k - nN) e^{2\pi j \frac{m_2}{M}k_2}  \right\}_{l \in \mathbb{N}_L, (m_1, m_2) \in \mathbb{N}_M^2, n \in \mathbb{Z}^2},
\]
where, $L,M,N$ are positive integers, $i,j$ are  the imaginary units in the quaternion algebra, and \( \{g_l\}_{l \in \mathbb{N}_L} \subset \ell^2(\mathbb{Z} \times \mathbb{Z}, \mathbb{H}) \). Special emphasis is placed on the case where the sequences \(g_l\) are real-valued. The questions addressed in this work include the characterization of quaternionic Gabor systems that form frames, the characterization of those that are orthonormal bases, and the admissibility of such systems. We also explore necessary and/or sufficient conditions for Gabor frames. The issue of duality is also discussed. Furthermore, we study the stability of these systems.}
	
	
	
	\maketitle
\section{Introduction and preliminaries}
Gabor systems play a central role in the field of harmonic analysis, particularly in the theory of frames and orthonormal bases. The study of quaternionic Gabor systems, which extend the classical notion of Gabor to the algebra of quaternions, is an emerging topic of growing importance due to their ability to model complex phenomena in higher-dimensional spaces. In this paper, we focus on the structure and properties of quaternionic Gabor systems in the context of the space \( \ell^2(\mathbb{Z} \times \mathbb{Z}, \mathbb{H}) \), where \( \mathbb{H} \) denotes the algebra of quaternions which is a four-dimensional real algebra with unity. In $\mathbb{H}$, $0$ denotes the null element  and $1$ denotes the identity with respect to multiplication. It also includes three so-called imaginary units, denoted by $i,j,k$. i.e.,
$$\mathbb{H}=\{a_0+a_1i+a_2j+a_3k:\; a_0,a_1,a_2,a_3\in \mathbb{R}\},$$
where $i^2=j^2=k^2=-1$, $ij=-ji=k$, $jk=-kj=i$ and $ki=-ik=j$. For each quaternion $q=a_0+a_1i+a_2j+a_3k$, we deifine the conjugate of $q$ denoted by $\overline{q}=a_0-a_1i-a_2j-a_2k \in \mathbb{H}$ and the module of $q$ denoted by $\vert q\vert $ as 
$$\vert q\vert =(\overline{q}q)^{\frac{1}{2}}=(q\overline{q})^{\frac{1}{2}}=\displaystyle{\sqrt{a_0^2+a_1^2+a_2^2+a_3^2}}.$$
For every $q\in \mathbb{H}$, $q^{-1}=\displaystyle{\frac{\overline{q}}{\vert q\vert ^2}}.$

Define \(\ell^2(\mathbb{Z} \times \mathbb{Z}, \mathbb{H}) := \{ f : \mathbb{Z} \times \mathbb{Z} \to \mathbb{H} : \displaystyle{\sum_{j \in \mathbb{Z}^2} |f(j)|^2} < \infty \}\) and \(\langle f, g \rangle := \displaystyle{\sum_{j \in \mathbb{Z}^2} \overline{f(j)} g(j)}\) for \(f, g \in \ell^2(\mathbb{Z} \times \mathbb{Z}, \mathbb{H})\). Then, \(\left( \ell^2(\mathbb{Z} \times \mathbb{Z}, \mathbb{H}), \langle \cdot, \cdot \rangle \right)\) is a right quaternionic Hilbert space with respect to right scalar multiplication. Denote by $\mathbb{B}(\ell^2(\mathbb{Z}\times \mathbb{Z},\mathbb{H})\,)$ the set of all right $\mathbb{H}$-linear bounded operators on $\ell^2(\mathbb{Z}\times \mathbb{Z},\mathbb{H})$.
Denote by $\ell_0(\mathbb{Z}\times \mathbb{Z},\mathbb{H})$ the subspace of $\ell^2(\mathbb{Z}\times \mathbb{Z},\mathbb{H})$ consisting of  sequences with finite support.

Denote by $\mathbb{N}_K:=\{0,1,\ldots,K-1\}$, where $K$ is a positive integer and  Let $L,M,N$ be positive integers. For all \(n=(n_1,n_2) \in \mathbb{Z}^2\), the translation operator with parameter \(nN:=(n_1N,n_2N)\) on \(\ell^2(\mathbb{Z} \times \mathbb{Z}, \mathbb{H})\), denoted by \(T_{nN}\), is the operator \(T_{nN} : \ell^2(\mathbb{Z} \times \mathbb{Z}, \mathbb{H}) \to \ell^2(\mathbb{Z} \times \mathbb{Z}, \mathbb{H})\), defined by \(T_{nN} f(k) = f(k - nN)\) for all \(f \in \ell^2(\mathbb{Z} \times \mathbb{Z}, \mathbb{H})\) and \(k \in \mathbb{Z}^2\). For all $m=(m_1,m_2)\in \mathbb{N}_M\times \mathbb{N}_M$, the modulation operator with parameter $\displaystyle{\frac{m}{M}:=(\frac{m_1}{M},\frac{m_2}{M})}$ on $\ell^2(\mathbb{Z} \times \mathbb{Z}, \mathbb{H})$, denoted by $E_{\frac{m}{M}}$, is the operator $E_{\frac{m}{M}}:\ell^2(\mathbb{Z} \times \mathbb{Z}, \mathbb{H})\rightarrow \ell^2(\mathbb{Z} \times \mathbb{Z}, \mathbb{H})$, defined by $E_{\frac{m}{M}}f(k)=e^{2\pi i \frac{m_1}{M}k_1} f(k) e^{2\pi j \frac{m_2}{M}k_2}$ for all $f\in \ell^2(\mathbb{Z} \times \mathbb{Z}, \mathbb{H})$ and $k=(k_1,k_2)\in \mathbb{Z}^2$. $T_{nN}\in \mathbb{B}(\ell^2(\mathbb{Z} \times \mathbb{Z}, \mathbb{H})\,)$ and is a unitary operator. $E_{\frac{m}{M}}$ is not right $\mathbb{H}$-linear (unlike to the complex case) but we still have $\|E_{\frac{m}{M}}f\|=\|f\|$, for all $f\in \ell^2(\mathbb{Z} \times \mathbb{Z}, \mathbb{H})$.

Given $\{g_l\}_{l\in \mathbb{N}_L}\subset \ell^2(\mathbb{Z}\times \mathbb{Z},\mathbb{H})$, the associated quaternionic Gabor system with parameters $M$ and $N$, denoted by $\mathcal{G}(g,L,M,N)$ is the system generated by the translations and modulations of the windows $g_l$, i.e., 
$$\mathcal{G}(g,L,M,N):=\{E_{\frac{m}{M}}T_{nN}g_l\}_{l\in \mathbb{N}_L,m\in \mathbb{N}_M^2,n\in \mathbb{Z}^2}.$$
Explicitly, $E_{\frac{m}{M}}T_{nN}g_l(k)=e^{2\pi i \frac{m_1}{M}k_1} g_l(k-nN) e^{2\pi j \frac{m_2}{M}k_2}$, for all $l\in \mathbb{N}_L, m=(m_1,m_2)\in \mathbb{N}_M^2, n\in \mathbb{Z}^2$ and $k=(k_1,k_2)\in \mathbb{Z}^2$. 

For all \( h \in \ell^2(\mathbb{Z} \times \mathbb{Z}, \mathbb{H}) \), we associate a function \( \mathcal{M}_h : \mathbb{Z}^2 \to \mathbb{M}(\mathbb{H}) \), where \( \mathbb{M}(\mathbb{H}) \) denotes the set of bi-infinite matrices indexed by pairs of integers, such that for all \( k \in \mathbb{Z}^2 \), the entries of the matrix \( \mathcal{M}_h(k) \) are given by \( (\mathcal{M}_h(k))_{p,n} = h(k + pM - nN) \), where \( p, n \in \mathbb{Z}^2 \). Note that the product, right scalar multiplication, and addition of matrices indexed by pairs are defined in a similar way to the case of usual matrices.	

For  a sequence $\{f(k)\}_{k\in \mathbb{Z}}$ with finite support, we define $\vert supp(f) \vert:=\left\vert \max(supp(f)\,)-\min( supp(f)\,)\right\vert$. Let $g\in \ell_0(\mathbb{Z}\times \mathbb{Z},\mathbb{H})$, we define $\vert supp(g)\vert$ as follows: 
$$\vert supp(g)\vert =\max\left\{ \max_{k_2\in \mathbb{Z}}\vert supp(\,g(.,k_2)\,)\vert\; ,\; \max_{k_1\in \mathbb{Z}}\vert supp(\,g(k_1,.)\,)\vert \right\}.$$

\begin{definition}[Gabor frames]
Let $\mathcal{G}(g,L,M,N)$ be a Gabor system  in $\ell^2(\mathbb{Z}\times \mathbb{Z},\mathbb{H})$. $\mathcal{G}(g,L,M,N)$ is said to be Gabor frame for $\ell^2(\mathbb{Z}\times \mathbb{Z},\mathbb{H})$,  if there exist $0<A\leq B<\infty$ such that for all $h\in \ell^2(\mathbb{Z}\times \mathbb{Z},\mathbb{H})$, the following inequality holds:
$$A\sum_{k\in \mathbb{Z}^2}\vert h(k)\vert^2\leq \displaystyle{\sum_{l\in \mathbb{N}_L}\sum_{n\in \mathbb{Z}^2}\sum_{m\in \mathbb{N}_M^2}\left\vert \langle E_{\frac{m}{M}}T_{nN}g_l,h\rangle \right\vert^2} \leq B\sum_{k\in \mathbb{Z}^2}\vert h(k)\vert^2.$$
\begin{enumerate}
\item If only the upper inequality holds, $\mathcal{G}(g,L,M,N)$ is called a Bessel Gabor system for $\ell^2(\mathbb{Z}\times \mathbb{Z},\mathbb{H})$.
\item If $A=B=1$, $\mathcal{G}(g,L,M,N)$ is called a Parseval Gabor frame for $\ell^2(\mathbb{Z}\times \mathbb{Z},\mathbb{H})$.\\
\end{enumerate}
\end{definition}

\begin{definition}
Let $\mathcal{G}(g,L,M,N)$ be a Bessel Gabor system  in $\ell^2(\mathbb{Z}\times \mathbb{Z},\mathbb{H})$. 
\begin{enumerate}
\item The pre-frame operator of $\mathcal{G}(g,L,M,N)$ is the right $\mathbb{H}$-linear bounded operator denoted by $T$ and  defined as follows: $$\begin{array}{rcl}
T:\ell^2(\mathbb{N}_L\times \mathbb{Z}^2\times \mathbb{N}_M^2,\mathbb{H})&\rightarrow& \ell^2(\mathbb{Z}\times \mathbb{Z},\mathbb{H})\\
q:=\{q_{l,n,m}\}_{l\in \mathbb{N}_L,n\in \mathbb{Z}^2,m\in \mathbb{N}_M^2}&\mapsto& \displaystyle{\sum_{l\in \mathbb{N}_L}\sum_{n\in \mathbb{Z}^2}\sum_{m\in \mathbb{N}_M^2}E_{\frac{m}{M}}T_{nN}g_l.\,q_{l,n,m}}. 
\end{array}$$
\item The transform opeartor of $\mathcal{G}(g,L,M,N)$, denoted by $\theta$, is the adjoint of its pre-frame operator. explicitly $\theta$ is defined as follows:
$$\begin{array}{rcl}
\theta:\ell^2(\mathbb{Z}\times \mathbb{Z},\mathbb{H}) &\rightarrow& \ell^2(\mathbb{N}_L\times \mathbb{Z}^2\times \mathbb{N}_M^2,\mathbb{H})\\
h&\mapsto&\{\langle E_{\frac{m}{M}}T_{nN}g_l,h\rangle\}_{l\in \mathbb{N}_L,n\in \mathbb{Z}^2,m\in \mathbb{N}_M^2}.
\end{array}$$
\item The frame operator of $\mathcal{G}(g,L,M,N)$, denoted by $S$, is the composite of $T$ and $\theta$. explicitly, $S$ is defined as follows: 
$$\begin{array}{rcl}
S:\ell^2(\mathbb{Z}\times \mathbb{Z},\mathbb{H})&\rightarrow& \ell^2(\mathbb{Z}\times \mathbb{Z},\mathbb{H})\\
h&\mapsto&\displaystyle{\sum_{l\in \mathbb{N}_L}\sum_{n\in \mathbb{Z}^2}\sum_{m\in \mathbb{N}_M^2} E_{\frac{m}{M}}T_{nN}g_l\,\langle E_{\frac{m}{M}}T_{nN}g_l,h\rangle}.
\end{array}$$
\end{enumerate}
\end{definition}	
\begin{proposition}\cite{11}
Let $\mathcal{G}(g,L,M,N)$ be a frame for $\ell^2(\mathbb{Z}\times \mathbb{Z},\mathbb{H})$ . Then:
\begin{enumerate}
    \item \( \theta \) is a right $\mathbb{H}$-linear  bounded injective operator.
    \item \( T \) is a right $\mathbb{H}$-linear bounded surjective operator.
    \item \( S \) is a right $\mathbb{H}$-linear bounded, positive, and invertible operator.\\
\end{enumerate}
\end{proposition}

In Section 2, we study discrete quaternionic Gabor systems and present necessary conditions for a quaternionic Gabor system to be a frame for \( \ell^2(\mathbb{Z} \times \mathbb{Z}, \mathbb{H}) \). We focus on the case where the windows take real values and characterize these quaternionic frames through an operator inequality. We also provide sufficient and necessary conditions for a system to be a Parseval quaternionic frame, and further describe the system using matrix-valued functions \( \mathcal{M}_{\cdot} \) associated with the system's windows.

In Section 3, we investigate quaternionic Gabor orthonormal bases for \( \ell^2(\mathbb{Z} \times \mathbb{Z}, \mathbb{H}) \). We characterize quaternionic Gabor systems that form orthonormal bases within the class of Parseval quaternionic frames and explore an interesting property regarding the number of windows in such a basis. We also characterize these systems using matrix-valued functions \( \mathcal{M}_{\cdot} \) and analyze their admissibility in \( \ell^2(\mathbb{Z} \times \mathbb{Z}, \mathbb{H}) \).

Section 4 is devoted to the study of the duality of Gabor frames with real-valued windows, a fundamental aspect in frame theory. Finally, in Section 5, we analyze the stability of Gabor frames with real-valued windows, a key topic for the robustness and reliability of these systems in practical applications.

This paper aims to deepen the understanding of quaternionic Gabor systems and provide new theoretical results that contribute to both frame theory and its applications in contexts where quaternions are a natural tool for modeling complex structures.

Note that all notations introduced in this introduction will be used consistently throughout, and we will not reintroduce them again. 
\section{Discrete quaternionic Gabor frames}	
In this section, we study discrete quaternionic Gabor systems. We first provide necessary conditions for a quaternionic Gabor system to be a frame for \( \ell^2(\mathbb{Z} \times \mathbb{Z}, \mathbb{H}) \). We then investigate these systems in the case where the windows take real values. In this case, we characterize quaternionic Gabor frames by an operator inequality, and give some sufficient and (or) necessary conditions for quaternionic Gabor frames,  and further describe the Parseval quaternionic Gabor frames  using matrix-valued functions \( \mathcal{M}_{\cdot} \) associated with the system's windows. At the end of the section, we conclude with an important theoretical example.

The following proposition gives a necessary condition for a quaternionic Gabor system to be frame.
\begin{proposition}\label{prop1}
Let $g:=\{g_l\}_{l\in \mathbb{N}_L}\subset \ell^2(\mathbb{Z}\times \mathbb{Z},\mathbb{H})$. If $\mathcal{G}(g,L,M,N)$ is a Gabor frame for $\ell^2(\mathbb{Z}\times \mathbb{Z},\mathbb{H})$ with frame bounds $A\leq B$, then for all $k\in \mathbb{Z}^2$, we have:
$$\frac{A}{M^2}\leq \sum_{l\in \mathbb{N}_L}\left(\mathcal{M}_{g_l}(k)\mathcal{M}_{g_l}^*(k)\right)_{0_{\mathbb{Z}^2},0_{\mathbb{Z}^2}}\leq \frac{B}{M^2}.$$
\end{proposition}	
\begin{proof}
Assume that $\mathcal{G}(g,L,M,N)$ is a Gabor frame for $\ell^2(\mathbb{Z}\times \mathbb{Z},\mathbb{H})$ with frame bounds $A\leq B$. Then for all $f\in \ell^2(\mathbb{Z}\times \mathbb{Z},\mathbb{H})$, we have:
$$\begin{array}{rcl}
A\displaystyle{\sum_{k\in \mathbb{Z}^2}\vert f(k)\vert^2} &\leq& \displaystyle{\sum_{l\in \mathbb{N}_L}\sum_{n\in \mathbb{Z}^2}\sum_{m\in \mathbb{N}_M^2}\left\vert\sum_{k\in \mathbb{Z}^2}e^{-2\pi j \frac{m_2}{M}k_2}\overline{g_l(k-nN)}e^{-2\pi i \frac{m_1}{M}k_1}f(k)\right\vert^2}\\
&\leq&B\displaystyle{\sum_{k\in \mathbb{Z}^2}\vert f(k)\vert^2}.
\end{array}$$
Fix $k'\in \mathbb{Z}^2$ and define $f:=\chi_{\{k'\}}$. Then,
$$\begin{array}{rcl}
&&\displaystyle{\sum_{l\in \mathbb{N}_L}\sum_{n\in \mathbb{Z}^2}\sum_{m\in \mathbb{N}_M^2}\left\vert\sum_{k\in \mathbb{Z}^2}e^{-2\pi j \frac{m_2}{M}k_2}\overline{g_l(k-nN)}e^{-2\pi i \frac{m_1}{M}k_1}f(k)\right\vert^2}\\
&=&\displaystyle{\sum_{l\in \mathbb{N}_L}\sum_{n\in \mathbb{Z}^2}\sum_{m\in \mathbb{N}_M^2}\left\vert e^{-2\pi j \frac{m_2}{M}k'_2}\overline{g_l(k'-nN)}e^{-2\pi i \frac{m_1}{M}k'_1}\right\vert^2}\\
&=& \displaystyle{\sum_{l\in \mathbb{N}_L}\sum_{n\in \mathbb{Z}^2}\sum_{m\in \mathbb{N}_M^2}\vert g_l(k'-nN)\vert^2}\\
&=&M^2\left(\displaystyle{\sum_{l\in \mathbb{N}_L}\mathcal{M}_{g_l}(k)\mathcal{M}_{g_l}^*(k)}\right)_{0_{\mathbb{Z}^2},0_{\mathbb{Z}^2}}.
\end{array}$$
Since $\displaystyle{\sum_{k\in \mathbb{Z}^2}\vert f(k)\vert^2=\vert f(k')\vert^2=1}$, the proof is, then, completed.
\end{proof}

The following proposition presents a necessary condition for a Gabor system to be Parseval frame using the system's parameters:
\begin{proposition}\label{prop}
Let $g:=\{g_l\}_{l\in \mathbb{N}_L}\subset \ell^2(\mathbb{Z}\times \mathbb{Z},\mathbb{H})$. If $\mathcal{G}(g,L,M,N)$ is a Parseval frame, then $N^2\leq LM^2$.
\end{proposition}

\begin{proof}
Assume $\mathcal{G}(g,L,M,N)$ is a Parseval farame for $\ell^2(\mathbb{Z}\times \mathbb{Z},\mathbb{H})$, then, by Proposition \ref{prop1}, $\displaystyle{\left(\sum_{l\in \mathbb{N}_L}\mathcal{M}_{g_l}(k)\mathcal{M}_{g_l}^*(k)\right)_{0_{\mathbb{Z}^2},0_{\mathbb{Z}^2}}=\frac{1}{M^2}}.$ Then, 
$$\displaystyle{\sum_{l\in \mathbb{N}_L}\sum_{n\in \mathbb{Z}^2}\vert g_l(k-nN)\vert^2=\frac{1}{M^2}}.$$
Then, 
$$\displaystyle{\sum_{l\in \mathbb{N}_L}\sum_{k\in \mathbb{N}_M^2}\sum_{n\in \mathbb{Z}^2}\vert g_l(k-nN)\vert^2=\frac{N^2}{M^2}}.$$
Hence, \begin{equation}
\displaystyle{\sum_{l\in \mathbb{N}_L}\| g_l\|^2=\frac{N^2}{M^2}}.\end{equation}Since $\|g_l\|=\|E_{\frac{m}{M}}T_{nN}g_l\|\leq 1$ for all $l\in \mathbb{N}_L$, where $n\in \mathbb{Z}^2$ and $m\in \mathbb{N}_M^2$ are arbitrary, then $\displaystyle{\frac{N^2}{M^2}}\leq L$.\\
\end{proof}

\begin{remark}
If \( S \) is the frame operator of a Gabor Bessel system \( \{ E_{\frac{m}{M}} T_{nN} g \}_{m \in \mathbb{N}_M^2, n \in \mathbb{Z}^2} \) in \( \ell^2(\mathbb{Z} \times \mathbb{Z}, \mathbb{H}) \), we do not necessarily have the property that \( S E_{\frac{q}{M}} T_{pN} = E_{\frac{q}{M}} T_{pN} S \) for any \( p \in \mathbb{Z}^2 \) and \( q \in \mathbb{N}_M^2 \), unlike the case of complex Gabor systems. This can be seen in the following example:

Let $\{E_{\frac{m}{M}}T_{nN}g\}_{m\in \mathbb{N}_M^2,n\in \mathbb{Z}^2}$ be a Bessel Gabor system and $S$ is its frame operator, $q=(1,1)$, $p=(0,0)$ and $h=\chi_{0_{\mathbb{Z}^2}}\in \ell^2(\mathbb{Z}\times \mathbb{Z}, \mathbb{H})$. We have:
$$\begin{array}{rcl}
\langle E_{\frac{m}{M}}T_{nN}g,E_{(\frac{1}{M},\frac{1}{M})}h\rangle&=&\displaystyle{\sum_{n\in \mathbb{Z}^2}\sum_{m\in \mathbb{N}_M^2}e^{ -2\pi j \frac{m_2}{M}k_2}\overline{g(k-nN)} e^{-2\pi i\frac{m_1}{M}k_1} e^{2\pi i \frac{1}{M}k_1}h(k)e^{2\pi j \frac{1}{M}k_2}}\\
&=& \overline{g(-nN)}.
\end{array}$$
Then, $SE_{(\frac{1}{M},\frac{1}{M})}h(k)=\displaystyle{\sum_{n\in \mathbb{Z}^2}\sum_{m\in \mathbb{N}_M^2} e^{2\pi \frac{m_1}{M}k_1}g(k-nN)e^{2\pi j \frac{m_2}{M}k_2}\; \overline{g(-nN)}}.$
For $k'=(0,1)$, we obtain:
$$SE_{(\frac{1}{M},\frac{1}{M})}h(k')=\displaystyle{\sum_{n\in \mathbb{Z}^2}\sum_{m\in \mathbb{N}_M^2} g(k-nN)e^{2\pi j \frac{m_2}{M}}\; \overline{g(-nN)}}.$$
On the other hand, we have:
$$\begin{array}{rcl}
\langle E_{\frac{m}{M}}T_{nN}g,h\rangle &=& \displaystyle{\sum_{k\in \mathbb{Z}^2}e^{-2\pi j \frac{m_2}{M}k_2}\overline{g(k-nN)}e^{-2\pi i \frac{m_1}{M}k_1} h(k)}\\
&=&\overline{g(-nN)}.
\end{array}$$
Then, $Sh=\displaystyle{\sum_{n\in \mathbb{Z}^2}\sum_{m\in \mathbb{N}_M^2}E_{\frac{m}{M}}T_{nN}g\;\overline{g(-nN)}}$. Then, $$\begin{array}{rcl}
E_{(\frac{1}{M},\frac{1}{M})}Sh(k)&=&e^{2\pi i \frac{1}{M}k_1}Sh(k)e^{2\pi j \frac{1}{M}k_2}\\
&=&e^{2\pi i \frac{1}{M}k_1}\left( \displaystyle{\sum_{n\in \mathbb{Z}^2}\sum_{m\in \mathbb{N}_M^2} e^{2\pi i \frac{m_1}{M}k_1}g(k-nN)e^{2\pi j \frac{m_2}{M}k_2} \overline{g(-nN)}}\right) e^{2\pi j \frac{1}{M}k_2}.
\end{array}$$
For $k'=(0,1)$, we obtain: $$
\begin{array}{rcl}
E_{(\frac{1}{M},\frac{1}{M})}Sh(k')&=&\left(\displaystyle{\sum_{n\in \mathbb{Z}^2}\sum_{m\in \mathbb{N}_M^2} g(k-nN)e^{2\pi j \frac{m_2}{M}} \overline{g(-nN)}}\right)   e^{2\pi j \frac{1}{M}}\\
&=& SE_{(\frac{1}{M},\frac{1}{M})}h(k')e^{2\pi j \frac{1}{M}} .
\end{array}$$
Hence, $SE_{(\frac{1}{M},\frac{1}{M})}h(k')\neq E_{(\frac{1}{M},\frac{1}{M})}Sh(k')$. Thus, $SE_{(\frac{1}{M},\frac{1}{M})}\neq E_{(\frac{1}{M},\frac{1}{M})}S$.\\
\end{remark}
\begin{remark}
It follows from the remark above that if \( \mathcal{G}(h,L,M,N) \) is a Gabor system in $\ell^2(\mathbb{Z}\times \mathbb{Z},\mathbb{H})$ and \( \mathcal{G}(g,L,M,N) \) is a Gabor Bessel system in $\ell^2(\mathbb{Z}\times \mathbb{Z},\mathbb{H})$ with frame operator \( S \), then \( S\left(\mathcal{G}(h,L,M,N)\right) \) is not necessarily a Gabor system in $\ell^2(\mathbb{Z}\times \mathbb{Z},\mathbb{H})$, unlike the case of complex Gabor systems.\\
\end{remark}

In what follows, for the sake of computational simplicity, we will focus on Gabor systems with windows that take real values. We present, first, some useful lemmas for the following.

The following lemma is obtained through simple computation.
\begin{lemma}\label{lem1}
For $k\in \mathbb{Z}$.
$$\sum_{m\in \mathbb{N}_M}e^{2\pi i \frac{m}{M}k }=\sum_{m\in \mathbb{N}_M}e^{2\pi j \frac{m}{M}k }=M\chi_{M\mathbb{Z}}(k).$$
\end{lemma}	

This lemma will be very useful for the remainder.
\begin{lemma}\label{lem3}	
Let $\{g_l\}_{l\in \mathbb{N}_L}\subset \ell^2(\mathbb{Z}\times \mathbb{Z},\mathbb{R})$ and $h\in \ell_0(\mathbb{Z}\times\mathbb{Z},\mathbb{H})$. then,
$$\displaystyle{\sum_{l\in \mathbb{N}_L}\sum_{n\in \mathbb{Z}^2}\sum_{m\in \mathbb{N}_M^2}\left\vert \langle E_{\frac{m}{M}}T_{nN}g_l,h\rangle \right\vert^2}=F_1(h)+F_2(h),$$
where $F_1(h):=M^2\displaystyle{\sum_{k\in \mathbb{Z}^2}\vert h(k)\vert^2 \left(\sum_{l\in \mathbb{N}_L}\mathcal{M}_{g_l}(k)\mathcal{M}_{g_l}^t(k)\right)_{0_{\mathbb{Z}^2},0_{\mathbb{Z}^2}},}$ and 
$F_2(h):=M^2\displaystyle{\sum_{k\in \mathbb{Z}^2}\sum_{0_{\mathbb{Z}^2}\neq p\in \mathbb{Z}^2} \overline{h(k)} h(k+pM)\left(\sum_{l\in \mathbb{N}_L}\mathcal{M}_{g_l}(k)\mathcal{M}_{g_l}^t(k)\right)_{0_{\mathbb{Z}^2},p}}$.\\
\end{lemma}	
\begin{proof}
In this proof, we will use the notation $a=(a_1,a_2)$ for every $a\in \mathbb{Z}^2$. Let $h\in \ell_0(\mathbb{Z}\times \mathbb{Z},\mathbb{H})$, we have:
$$\begin{array}{rcl}
&&\displaystyle{\sum_{l\in \mathbb{N}_L}\sum_{n\in \mathbb{Z}^2}\sum_{m\in \mathbb{N}_M^2}\left\vert \langle E_{\frac{m}{M}}T_{nN}g_l,h\rangle \right\vert^2}\\
&=&\displaystyle{\sum_{l\in \mathbb{N}_L}\sum_{n\in \mathbb{Z}^2}\sum_{m\in \mathbb{N}_M^2}\left\vert \sum_{k\in \mathbb{Z}^2}\overline{E_{\frac{m}{M}}T_{nN}g_l(k)} h(k)\right\vert^2}\\
&=&\displaystyle{\sum_{l\in \mathbb{N}_L}\sum_{n\in \mathbb{Z}^2}\sum_{m\in \mathbb{N}_M^2}\left(\sum_{k\in \mathbb{Z}^2}\overline{h(k)}E_{\frac{m}{M}}T_{nN}g_l(k)\right)\left(\sum_{k\in \mathbb{Z}^2}\overline{E_{\frac{m}{M}}T_{nN}g_l(k)} h(k)\right)}\\
\end{array}$$
$$\begin{array}{rcl}
&=&\displaystyle{\sum_{l\in \mathbb{N}_L}\sum_{n\in \mathbb{Z}^2}\sum_{m\in \mathbb{N}_M^2}\sum_{k\in \mathbb{Z}^2}\left\vert \overline{E_{\frac{m}{M}}T_{nN}g_l(k)}h(k)\right\vert^2}\\
&+&\displaystyle{\sum_{l\in \mathbb{N}_L}\sum_{n\in \mathbb{Z}^2}\sum_{m\in \mathbb{N}_M^2}\sum_{k\in \mathbb{Z}^2}\sum_{p_1=k_1}\sum_{p_2\neq k_2}\overline{h(k)}E_{\frac{m}{M}}T_{nN}g_l(k)\overline{E_{\frac{m}{M}}T_{nN}g_l(p_1,p_2)}h(p_1,p_2)}\\
&+&\displaystyle{\sum_{l\in \mathbb{N}_L}\sum_{n\in \mathbb{Z}^2}\sum_{m\in \mathbb{N}_M^2}\sum_{k\in \mathbb{Z}^2}\sum_{p_1\neq k_1}\sum_{p_2\in \mathbb{Z}}\overline{h(k)}E_{\frac{m}{M}}T_{nN}g_l(k)\overline{E_{\frac{m}{M}}T_{nN}g_l(p_1,p_2)}h(p_1,p_2)}.
\end{array}$$
We have, $$\begin{array}{rcl}
&&\displaystyle{\sum_{l\in \mathbb{N}_L}\sum_{n\in \mathbb{Z}^2}\sum_{m\in \mathbb{N}_M^2}\sum_{k\in \mathbb{Z}^2}\left\vert \overline{E_{\frac{m}{M}}T_{nN}g_l(k)}h(k)\right\vert^2}\\
&=&\displaystyle{\sum_{l\in \mathbb{N}_L}\sum_{n\in \mathbb{Z}^2}\sum_{m\in \mathbb{N}_M^2}\sum_{k\in \mathbb{Z}^2}\left\vert e^{-2\pi j\frac{m_2}{M}k_2}\overline{g_l(k-nN)}e^{-2\pi i \frac{m_1}{M}k_1} h(k)\right\vert^2} \\
&=& \displaystyle{\sum_{l\in \mathbb{N}_L}\sum_{n\in \mathbb{Z}^2}\sum_{m\in \mathbb{N}_M^2}\sum_{k\in \mathbb{Z}^2}\left\vert h(k)\right\vert^2\left\vert g_l(k-nN)\right\vert^2}\\
&=&M^2\displaystyle{\sum_{k\in \mathbb{Z}^2} \left\vert h(k)\right\vert^2 \sum_{l\in \mathbb{N}_L}\sum_{n\in \mathbb{Z}^2}\left\vert g(k-nN)\right\vert^2}\\
&=& M^2\displaystyle{\sum_{k\in \mathbb{Z}^2} \left\vert h(k)\right\vert^2 \left(\sum_{l\in \mathbb{N}_L}\mathcal{M}_{g_l}(k)\mathcal{M}_{g_l}^t(k)\right)_{0_{\mathbb{Z}^2},0_{\mathbb{Z}^2}}}\\
&=&F_1(h).
\end{array}$$
The above change of summation is justified by the fact that $h\in \ell_0(\mathbb{Z}\times \mathbb{Z},\mathbb{H})$.\\
And we have, 
$$\begin{array}{rcl}
&&\displaystyle{\sum_{l\in \mathbb{N}_L}\sum_{n\in \mathbb{Z}^2}\sum_{m\in \mathbb{N}_M^2}\sum_{k\in \mathbb{Z}^2}\sum_{p_1=k_1}\sum_{p_2\neq k_2}\overline{h(k)}E_{\frac{m}{M}}T_{nN}g_l(k)\overline{E_{\frac{m}{M}}T_{nN}g_l(p_1,p_2)}h(p_1,p_2)}\\
&=&\displaystyle{\sum_{l\in \mathbb{N}_L}\sum_{n\in \mathbb{Z}^2}\sum_{m\in \mathbb{N}_M^2}\sum_{k\in \mathbb{Z}^2}\sum_{p_1=k_1}\sum_{p_2\neq k_2}\overline{h(k)}e^{2\pi i \frac{m_1}{M}k_1}g_l(k-nN)e^{2\pi j \frac{m_2}{M}k_2} e^{-2\pi j\frac{m_2}{M}p_2}}\\
&&g_l((k_1,p_2)-nN)e^{-2\pi i \frac{m_1}{M}k_1}h(k_1,p_2)\\
&=&M^2\displaystyle{\sum_{k\in \mathbb{Z}^2}\sum_{q\neq 0} \overline{h(k)}h(k_1,k_2+qM)\sum_{l\in \mathbb{N}_L}\sum_{n\in \mathbb{Z}^2}g_l(k-nN)g_l(k+(0,q)M-nN)}\\
&=&M^2\displaystyle{\sum_{k\in \mathbb{Z}^2}\sum_{q\neq 0} \overline{h(k)}h(k_1,k_2+qM)\left(\sum_{l\in \mathbb{N}_L}\mathcal{M}_{g_l}(k)\mathcal{M}_{g_l}^t(k)\right)_{0_{\mathbb{Z}^2},(0,q)}}.
\end{array}$$
The above change of summation is justified by the fact that \( h \in \ell_0(\mathbb{Z} \times \mathbb{Z}, \mathbb{H}) \), the commutativity of \( g_l \) with the other terms is justified by the fact that it takes real values, and the fact that \( p_2 = k_2 + qM \) is justified by Lemma \ref{lem1}.\\
We also have: 
$$\begin{array}{rcl}
&&\displaystyle{\sum_{l\in \mathbb{N}_L}\sum_{n\in \mathbb{Z}^2}\sum_{m\in \mathbb{N}_M^2}\sum_{k\in \mathbb{Z}^2}\sum_{p_1\neq k_1}\sum_{p_2\in \mathbb{Z}}\overline{h(k)}E_{\frac{m}{M}}T_{nN}g_l(k)\overline{E_{\frac{m}{M}}T_{nN}g_l(p_1,p_2)}h(p_1,p_2)}\\
&=&\displaystyle{\sum_{l\in \mathbb{N}_L}\sum_{n\in \mathbb{Z}^2}\sum_{m\in \mathbb{N}_M^2}\sum_{k\in \mathbb{Z}^2}\sum_{p_1\neq k_1}\sum_{p_2\in \mathbb{Z}}\overline{h(k)} e^{2\pi i \frac{m_1}{M}k_1}g_l(k-nN)e^{2\pi j\frac{m_2}{M}k_2}g_l(k-nN)e^{2\pi j\frac{m_2}{M}k_2}}\\
&&e^{-2\pi j \frac{m_2}{M}p_2}
g_l(p-nN)e^{-2\pi i \frac{m_1}{M}p_1} h(p_1,p_2)\\
&=&M^2\displaystyle{\sum_{k\in \mathbb{Z}^2}\sum_{q_1\neq 0}\sum_{q_2\in \mathbb{Z}}\overline{h(k)}h(k+(q_1,q_2)M)\sum_{l\in \mathbb{N}_L}\sum_{n\in \mathbb{Z}^2}g_l(k-nN)g_l(k+(q_1,q_2)M-nN)}\\
&=&M^2\displaystyle{\sum_{k\in \mathbb{Z}^2}\sum_{q_1\neq 0}\sum_{q_2\in \mathbb{Z}}\overline{h(k)}h(k+(q_1,q_2)M)\left( \sum_{l\in \mathbb{N}_L}\mathcal{M}_{g_l}(k)\mathcal{M}_{g_l}^t(k)\right)_{0_{\mathbb{Z}^2},(q_1,q_2)}}.
\end{array}$$
The above change of summation is justified by the fact that \( h \in \ell_0(\mathbb{Z} \times \mathbb{Z}, \mathbb{H}) \), the commutativity of \( g_l \) with the other terms is justified by the fact that it takes real values, and the fact that \(p_1=k_1+q_1M\) and  \( p_2 = k_2 + 2M \) is justified by Lemma \ref{lem1}.\\
Hence, $$\begin{array}{rcl}
&&\displaystyle{\sum_{l\in \mathbb{N}_L}\sum_{n\in \mathbb{Z}^2}\sum_{m\in \mathbb{N}_M^2}\sum_{k\in \mathbb{Z}^2}\sum_{p_1=k_1}\sum_{p_2\neq k_2}\overline{h(k)}E_{\frac{m}{M}}T_{nN}g_l(k)\overline{E_{\frac{m}{M}}T_{nN}g_l(p_1,p_2)}h(p_1,p_2)}\\
&+&\displaystyle{\sum_{l\in \mathbb{N}_L}\sum_{n\in \mathbb{Z}^2}\sum_{m\in \mathbb{N}_M^2}\sum_{k\in \mathbb{Z}^2}\sum_{p_1\neq k_1}\sum_{p_2\in \mathbb{Z}}\overline{h(k)}E_{\frac{m}{M}}T_{nN}g_l(k)\overline{E_{\frac{m}{M}}T_{nN}g_l(p_1,p_2)}h(p_1,p_2)}\\
&=&M^2\displaystyle{\sum_{k\in \mathbb{Z}^2}\sum_{0_{\mathbb{Z}^2}\neq p\in \mathbb{Z}^2} \overline{h(k)} h(k+pM)\left(\sum_{l\in \mathbb{N}_L}\mathcal{M}_{g_l}(k)\mathcal{M}_{g_l}^t(k)\right)_{0_{\mathbb{Z}^2},p}}\\
&=&F_2(h).
\end{array}$$
\end{proof}

\begin{lemma}\label{lem5}
Let $g,h\in \ell^2(\mathbb{Z}\times \mathbb{Z},\mathbb{H})$. Then, $\mathcal{M}_{g}(.)\mathcal{M}_h^*(.)$ is $N\mathbb{Z}^2-$periodic, i.e., for all $k,q\in \mathbb{Z}^2$, $\mathcal{M}_g(k+qN)\mathcal{M}_h^*(k+qN)=\mathcal{M}_g(k)\mathcal{M}_h^*(k).$\\
\end{lemma}	
\begin{proof}
Let $k,q\in \mathbb{Z}^2$ and $m,p\in \mathbb{Z}^2$, we have:
$$\begin{array}{rcl}
\left(\mathcal{M}_g(k+qN)\mathcal{M}_h^*(k+qN)\right)_{m,p}&=&\displaystyle{\sum_{n\in \mathbb{Z}^2}g(k+qN+mM-nN)\overline{h(k+qN+pM-nN)}}\\
&=&\displaystyle{\sum_{n\in \mathbb{Z}^2}g(k+mM-(n-q)N)\overline{h(k+pM-(n-q)N)}}\\
&=&\displaystyle{\sum_{n\in \mathbb{Z}^2}g(k+mM-nN)\overline{h(k+pM-nN)}}\\
&=&\left(\mathcal{M}_g(k)\mathcal{M}_h^*(k)\right)_{m,p}.
\end{array}$$
\end{proof}
The following theorem characterizes discrete quaternionic Gabor frames with real-values windows.
\begin{theorem}
Let $g:=\{g_l\}_{l\in \mathbb{N}_L}\subset \ell^2(\mathbb{Z}\times \mathbb{Z},\mathbb{R})$. Then, the following statements are equivalent:
\begin{enumerate}
\item $\mathcal{G}(g,L,M,N)$ is a frame for $\ell^2(\mathbb{Z}\times \mathbb{Z},\mathbb{H})$ with frame bounds $A\leq B$.
\item For all $k\in \mathbb{Z}^2$, $$\frac{A}{M^2}I\preceq \sum_{l\in \mathbb{N}_L}\mathcal{M}_{g_l}(k)\mathcal{M}_{g_l}^t(k)\preceq \frac{B}{M^2}I.$$
\item For all $k\in \mathbb{N}_N^2$, $$\frac{A}{M^2}I\preceq \sum_{l\in \mathbb{N}_L}\mathcal{M}_{g_l}(k)\mathcal{M}_{g_l}^t(k)\preceq \frac{B}{M^2}I.$$
\end{enumerate}
Where $I$ denotes the identity matrix, i.e, $I=\left(\delta_{p,n}\right)_{p,n}$.
\end{theorem}
\begin{proof}
By a simple computation and using Lemma \ref{lem2}, we obtain:
\begin{equation}
\begin{array}{rcl}
&&\displaystyle{\sum_{l\in \mathbb{N}_L}\sum_{n\in \mathbb{Z}^2}\sum_{m\in \mathbb{N}_M^2}\vert \langle E_{\frac{m}{M}}T_{nN}g_l,f\rangle\vert^2}\\
&=&\displaystyle{\sum_{l\in \mathbb{N}_L}\sum_{n\in \mathbb{Z}^2}\sum_{m\in \mathbb{N}_M^2}\left\vert \sum_{k\in \mathbb{Z}^2}e^{-2\pi j \frac{m_2}{M}k_2}\overline{g(k-nN)}e^{-2\pi i \frac{m_1}{M}k_1}f(k)\right\vert^2}\\
&=&\displaystyle{\sum_{l\in \mathbb{N}_L}\sum_{n\in \mathbb{Z}^2}\sum_{m\in \mathbb{N}_M^2}\left\vert \sum_{k\in \mathbb{N}_M^2} e^{-2\pi j \frac{m_2}{M}k_2}g(k-nN)e^{-2\pi i \frac{m_1}{M}k_1}\left(\sum_{p\in \mathbb{Z}^2}g(k+pM-nN)f(k+pM)\right)\right\vert^2}\\
&=&M^2\displaystyle{\sum_{l\in \mathbb{N}_L}\sum_{n\in \mathbb{Z}^2}\sum_{k\in \mathbb{N}_M^2}\left\vert \sum_{p\in \mathbb{Z}^2}g(k+pM-nN)f(k+pM)\right\vert^2}\\
\end{array}
\end{equation}
Assume that $\mathcal{G}(g,L,M,N)$ is a frame with frame bounds $A\leq B$. Then, for all $f\in \ell^2(\mathbb{Z}\times \mathbb{Z},\mathbb{H})$, we have:
$$\begin{array}{rcl}
A\displaystyle{\sum_{k\in \mathbb{Z}^2}\vert f(k)\vert^2}&\leq& M^2\displaystyle{\sum_{l\in \mathbb{N}_L}\sum_{n\in \mathbb{Z}^2}\sum_{k\in \mathbb{N}_M^2}\left\vert \sum_{p\in \mathbb{Z}^2}g(k+pM-nN)f(k+pM)\right\vert^2}\\
&\leq& B\displaystyle{\sum_{k\in \mathbb{Z}^2}\vert f(k)\vert^2}.
\end{array}$$
Fix $k'\in \mathbb{N}_M^2$ and define a $M$-periodic sequence $\phi$ defined on $\mathbb{N}_M^2$ by $\phi=\chi_{\{k'\}}$. Then replacing $f$ by $f\phi$ in the above inequality leads to:
$$\begin{array}{rcl}
A\displaystyle{\sum_{p\in \mathbb{Z}^2}\vert f(k'+pM)\vert^2}&\leq& M^2\displaystyle{\sum_{l\in \mathbb{N}_L}\sum_{n\in \mathbb{Z}^2}\left\vert \sum_{p\in \mathbb{Z}^2}g(k'+pM-nN)f(k'+pM)\right\vert^2}\\
&\leq& B\displaystyle{\sum_{p\in \mathbb{Z}^2}\vert f(k'+pM)\vert^2}.
\end{array}$$
Then, by arbitrariness of $k'$ and by $M$-periodicity of the terms in the inequality, we have: 
$$\begin{array}{rcl}
A\displaystyle{\sum_{p\in \mathbb{Z}^2}\vert f(.+pM)\vert^2}&\leq& M^2\displaystyle{\sum_{l\in \mathbb{N}_L}\sum_{n\in \mathbb{Z}^2}\left\vert \sum_{p\in \mathbb{Z}^2}g(.+pM-nN)f(.+pM)\right\vert^2}\\
&\leq& B\displaystyle{\sum_{p\in \mathbb{Z}^2}\vert f(.+pM)\vert^2}.
\end{array}$$
Let, now, $x\in \ell^2(\mathbb{Z}\times \mathbb{Z},\mathbb{H})$ be arbitrary, fix $k'\in \mathbb{N}_N^2$, and define $f\ell^2(\mathbb{Z}\times \mathbb{Z},\mathbb{H})$ by $f(k+pM)=x(p)\chi_{\{k'\}}(k)$ for all $k\in \mathbb{N}_M^2$ and $p\in \mathbb{Z}^2$. Then, 
$$\begin{array}{rcl}
A\displaystyle{\sum_{p\in \mathbb{Z}^2}\vert x(p)\vert^2}&\leq& M^2\displaystyle{\sum_{l\in \mathbb{N}_L}\sum_{n\in \mathbb{Z}^2}\left\vert \sum_{p\in \mathbb{Z}^2}g(k'+pM-nN)x(p)\right\vert^2} \leq B\displaystyle{\sum_{p\in \mathbb{Z}^2}\vert x(p)\vert^2}.
\end{array}$$
Hence, $$\begin{array}{rcl}
A\langle x,x\rangle &\leq& M^2\displaystyle{\sum_{l\in \mathbb{N}_L}\sum_{n\in \mathbb{Z}^2}\left\vert \left(\mathcal{M}_{g_l}^t(k')x\right)_{n}\right\vert^2} \leq B\langle x,x\rangle.
\end{array}$$
Thus, $$\begin{array}{rcl}
A\langle x,x\rangle &\leq& M^2\displaystyle{\sum_{l\in \mathbb{N}_L}\left\| \mathcal{M}_{g_l}^t(k')x\right\|^2} \leq B\langle x,x\rangle.
\end{array}$$
i.e., $$\begin{array}{rcl}
A\langle x,x\rangle &\leq& M^2\displaystyle{\left\langle \sum_{l\in \mathbb{N}_L}\mathcal{M}_{g_l}(k')\mathcal{M}_{g_l}^t(k')x,x\right\rangle} \leq B\langle x,x\rangle.
\end{array}$$
Hence, by arbitrariness of $k'$ in $\mathbb{N}_M^2$, we have for all $k\in \mathbb{N}_N^2$, $$\frac{A}{M^2}I\preceq \sum_{l\in \mathbb{N}_L}\mathcal{M}_{g_l}(k)\mathcal{M}_{g_l}^t(k)\preceq \frac{B}{M^2}I.$$
Conversely, assume $3.$, and let $f\in \ell^2(\mathbb{Z}\times \mathbb{Z},\mathbb{H})$ and $k\in \mathbb{N}_M^2$. Define $x^{(k)}\in \ell^2(\mathbb{Z}\times \mathbb{Z},\mathbb{H})$ by $x^{(k)}(p)=f(k+pM)$ for all $p\in \mathbb{Z}^2$. Then, by $3$., we have (for all $k\in \mathbb{N}_M^2$): 
$$\frac{A}{M^2}\|x^{(k)}\|^2\leq \sum_{l\in \mathbb{N}_L}\left\|\mathcal{M}_{g_l}^t(k)x^{(k)}\right\|^2\leq \frac{B}{M^2}\|x^{(k)}\|^2.$$
i.e., $$\frac{A}{M^2}\sum_{p\in \mathbb{Z}^2}\vert f(k+pM)\vert^2\leq \sum_{l\in \mathbb{N}_L}\sum_{n\in \mathbb{Z}^2}\left\vert\sum_{p\in \mathbb{Z}^2}g(k+pM-nN)f(k+pM)\right\vert^2\leq \frac{B}{M^2}\sum_{p\in \mathbb{Z}^2}\vert f(k+pM)\vert^2.$$
Then,
$$\begin{array}{rcl}
\displaystyle{\frac{A}{M^2}\sum_{k\in \mathbb{N}_M^2}\sum_{p\in \mathbb{Z}^2}\vert f(k+pM)\vert^2}&\leq& \displaystyle{\sum_{l\in \mathbb{N}_L}\sum_{n\in \mathbb{Z}^2}\sum_{k\in \mathbb{N}_M^2}\left\vert\sum_{p\in \mathbb{Z}^2}g(k+pM-nN)f(k+pM)\right\vert^2}\\
&\leq &\displaystyle{\frac{B}{M^2}\sum_{k\in \mathbb{N}_M^2}\sum_{p\in \mathbb{Z}^2}\vert f(k+pM)\vert^2}.
\end{array}$$
Then, by $(2)$ and by arbitrariness of $f$, we deduce that for all $f\in \ell^2(\mathbb{Z}\times \mathbb{Z},\mathbb{H})$:
$$A\|f\|^2\leq \sum_{l\in \mathbb{N}_L}\sum_{n\in \mathbb{Z}^2}\sum_{m\in \mathbb{N}_M^2}\left\vert \langle E_{\frac{m}{M}}T_{nN}g_l,f\rangle \right\vert^2\leq B\| f\|^2.$$
Hence, $\mathcal{G}(g,L,M,N)$ is a frame with frame bounds $A\leq B$. It is clear that $2.$ and $3.$ are equivalent by Lemma \ref{lem5}.
\end{proof}	

In the following theorem, we provide sufficient conditions for a Gabor system to be Bessel or, moreover, a Gabor frame.
\begin{theorem}\label{thmd}
Let $g:=\{g_l\}_{l\in \mathbb{N}_L}\subset \ell^2(\mathbb{Z}\times \mathbb{Z},\mathbb{R})$. If
$$B:=M^2 \max_{k\in \mathbb{N}_N^2}\sum_{p\in \mathbb{Z}^2}\left\vert \left(\sum_{l\in \mathbb{N}_L}\mathcal{M}_{g_l}(k)\mathcal{M}_{g_l}^t(k)\right)_{0_{\mathbb{Z}^2},p}\right\vert< \infty,$$
then $\mathcal{G}(g,L,M,N)$ is a Bessel Gabor system for $\ell^2(\mathbb{Z}\times \mathbb{Z},\mathbb{H})$ with Bessel bound $B$. If also 
$$A:=M^2 \min_{k\in \mathbb{N}_N^2}\sum_{l\in \mathbb{N}_L}\left[ (\mathcal{M}_{g_l}(k)\mathcal{M}_{g_l}^t(k))_{0_{\mathbb{Z}^2},0_{\mathbb{Z}^2}}-\sum_{0_{\mathbb{Z}^2}\neq p\in \mathbb{Z}^2}\left\vert (\mathcal{M}_{g_l}(k)\mathcal{M}_{g_l}^t(k))_{0_{\mathbb{Z}^2},p}\right\vert \right]> 0,$$
then $\mathcal{G}(g,L,M,N)$ is a Gabor frame for $\ell^2(\mathbb{Z}\times \mathbb{Z},\mathbb{H})$ with frame bounds $A\leq B$.
\end{theorem}	
 
\begin{proof}
By Lemma \ref{lem3}, we have for all $h\in \ell_0(\mathbb{Z}\times \mathbb{Z},\mathbb{H})$:
\begin{equation}
\begin{array}{rcl}
\displaystyle{\sum_{l\in \mathbb{N}_L}\sum_{n\in \mathbb{Z}^2}\sum_{m\in \mathbb{N}_M^2}\left\vert \langle E_{\frac{m}{M}}T_{nN}g_l,h\rangle \right\vert^2}&=&M^2\displaystyle{\sum_{k\in \mathbb{Z}^2}\vert h(k)\vert^2 \left(\sum_{l\in \mathbb{N}_L}\mathcal{M}_{g_l}(k)\mathcal{M}_{g_l}^t(k)\right)_{0_{\mathbb{Z}^2},0_{\mathbb{Z}^2}}}\\
&+&M^2\displaystyle{\sum_{k\in \mathbb{Z}^2}\sum_{0_{\mathbb{Z}^2}\neq p\in \mathbb{Z}^2} \overline{h(k)} h(k+pM)\left(\sum_{l\in \mathbb{N}_L}\mathcal{M}_{g_l}(k)\mathcal{M}_{g_l}^t(k)\right)_{0_{\mathbb{Z}^2},p}}.
\end{array}
\end{equation}
By applying the triangle inequality and then the Cauchy-Schwarz inequality, we obtain:
$$\begin{array}{rcl}
&&\left \vert\displaystyle{\sum_{k\in \mathbb{Z}^2}\sum_{0_{\mathbb{Z}^2}\neq p\in \mathbb{Z}^2} \overline{h(k)} h(k+pM)\left(\sum_{l\in \mathbb{N}_L}\mathcal{M}_{g_l}(k)\mathcal{M}_{g_l}^t(k)\right)_{0_{\mathbb{Z}^2},p}}\right\vert\\
&\leq& \displaystyle{\sum_{k\in \mathbb{Z}^2}\sum_{0_{\mathbb{Z}^2}\neq p\in \mathbb{Z}^2} \vert \overline{h(k)}\vert \vert h(k+pM)\vert \left\vert\left(\sum_{l\in \mathbb{N}_L}\mathcal{M}_{g_l}(k)\mathcal{M}_{g_l}^t(k)\right)_{0_{\mathbb{Z}^2},p}\right\vert}\\
&\leq& \left(\displaystyle{\sum_{k\in \mathbb{Z}^2}\sum_{0_{\mathbb{Z}^2}\neq p\in \mathbb{Z}^2}\vert h(k)\vert^2 \left\vert\left(\sum_{l\in \mathbb{N}_L}\mathcal{M}_{g_l}(k)\mathcal{M}_{g_l}^t(k)\right)_{0_{\mathbb{Z}^2},p}\right\vert}\right)^{\frac{1}{2}}\\
&\times& \left(\displaystyle{\sum_{k\in \mathbb{Z}^2}\sum_{0_{\mathbb{Z}^2}\neq p\in \mathbb{Z}^2}\vert h(k+pM)\vert^2 \left\vert\left(\sum_{l\in \mathbb{N}_L}\mathcal{M}_{g_l}(k)\mathcal{M}_{g_l}^t(k)\right)_{0_{\mathbb{Z}^2},p}\right\vert}\right)^{\frac{1}{2}}.
\end{array}$$
We also have that: 
$$\begin{array}{rcl}
&&\displaystyle{\sum_{k\in \mathbb{Z}^2}\sum_{0_{\mathbb{Z}^2}\neq p\in \mathbb{Z}^2}\vert h(k+pM)\vert^2 \left\vert\left(\sum_{l\in \mathbb{N}_L}\mathcal{M}_{g_l}(k)\mathcal{M}_{g_l}^t(k)\right)_{0_{\mathbb{Z}^2},p}\right\vert}\\
&=&\displaystyle{\sum_{k\in \mathbb{Z}^2}\sum_{0_{\mathbb{Z}^2}\neq p\in \mathbb{Z}^2}\vert h(k)\vert^2 \left\vert\left(\sum_{l\in \mathbb{N}_L}\mathcal{M}_{g_l}(k)\mathcal{M}_{g_l}^t(k)\right)_{-p,0_{\mathbb{Z}^2}}\right\vert}\\
&=&\displaystyle{\sum_{k\in \mathbb{Z}^2}\sum_{0_{\mathbb{Z}^2}\neq p\in \mathbb{Z}^2}\vert h(k)\vert^2 \left\vert\left(\sum_{l\in \mathbb{N}_L}\mathcal{M}_{g_l}(k)\mathcal{M}_{g_l}^t(k)\right)_{p,0}\right\vert}\\
&=&\displaystyle{\sum_{k\in \mathbb{Z}^2}\sum_{0_{\mathbb{Z}^2}\neq p\in \mathbb{Z}^2}\vert h(k)\vert^2 \left\vert\left(\sum_{l\in \mathbb{N}_L}\mathcal{M}_{g_l}(k)\mathcal{M}_{g_l}^t(k)\right)_{0_{\mathbb{Z}^2},p}\right\vert}.
\end{array}$$
It follows that:
\begin{equation}
\begin{array}{rcl}
&&\left \vert\displaystyle{\sum_{k\in \mathbb{Z}^2}\sum_{0_{\mathbb{Z}^2}\neq p\in \mathbb{Z}^2} \overline{h(k)} h(k+pM)\left(\sum_{l\in \mathbb{N}_L}\mathcal{M}_{g_l}(k)\mathcal{M}_{g_l}^t(k)\right)_{0_{\mathbb{Z}^2},p}}\right\vert\\
&\leq&\displaystyle{\sum_{k\in \mathbb{Z}^2}\sum_{0_{\mathbb{Z}^2}\neq p\in \mathbb{Z}^2}\vert h(k)\vert^2 \left\vert\left(\sum_{l\in \mathbb{N}_L}\mathcal{M}_{g_l}(k)\mathcal{M}_{g_l}^t(k)\right)_{0_{\mathbb{Z}^2},p}\right\vert}.
\end{array}
\end{equation}
Then , combining $(3)$, $(4)$ and the hypothesis leads to: 
$$\begin{array}{rcl}
&&\displaystyle{\sum_{l\in \mathbb{N}_L}\sum_{n\in \mathbb{Z}^2}\sum_{m\in \mathbb{N}_M^2}\left\vert \langle E_{\frac{m}{M}}T_{nN}g_l,h\rangle \right\vert^2}\\
&\leq& M^2\displaystyle{\sum_{k\in \mathbb{Z}^2}\sum_{ p\in \mathbb{Z}^2}\vert h(k)\vert^2 \left\vert\left(\sum_{l\in \mathbb{N}_L}\mathcal{M}_{g_l}(k)\mathcal{M}_{g_l}^t(k)\right)_{0_{\mathbb{Z}^2},p}\right\vert}\\
&\leq&B\|h\|^2,
\end{array}$$
this for all $h\in \ell_0(\mathbb{Z}\times \mathbb{Z},\mathbb{H})$ which is dense in $\ell^2(\mathbb{Z}\times \mathbb{Z},\mathbb{H})$. Hence, $\mathcal{G}(g,L,M,N)$ is a Bessel Gabor system with Bessel bound $B$. Assume, now, that: 
$$A:=M^2 \min_{k\in \mathbb{N}_N^2}\sum_{l\in \mathbb{N}_L}\left[ (\mathcal{M}_{g_l}(k)\mathcal{M}_{g_l}^t(k))_{0_{\mathbb{Z}^2},0_{\mathbb{Z}^2}}-\sum_{0_{\mathbb{Z}^2}\neq p\in \mathbb{Z}^2}\left\vert (\mathcal{M}_{g_l}(k)\mathcal{M}_{g_l}^t(k))_{0_{\mathbb{Z}^2},p}\right\vert \right]> 0,$$
By Lemma \ref{lem3}, we have for all $h\in \ell_0(\mathbb{Z}\times \mathbb{Z},\mathbb{H})$:
$$\begin{array}{rcl}
&&\displaystyle{\sum_{l\in \mathbb{N}_L}\sum_{n\in \mathbb{Z}^2}\sum_{m\in \mathbb{N}_M^2}\left\vert \langle E_{\frac{m}{M}}T_{nN}g_l,h\rangle \right\vert^2}\\
&\geqslant& M^2\displaystyle{\sum_{k\in \mathbb{Z}^2}\vert h(k)\vert^2 \left(\sum_{l\in \mathbb{N}_L}\mathcal{M}_{g_l}(k)\mathcal{M}_{g_l}^t(k)\right)_{0_{\mathbb{Z}^2},0_{\mathbb{Z}^2}}}\\
&-&M^2\left\vert\displaystyle{\sum_{k\in \mathbb{Z}^2}\sum_{0_{\mathbb{Z}^2}\neq p\in \mathbb{Z}^2}\overline{h(k)}h(k+pM) \left(\sum_{l\in \mathbb{N}_L}\mathcal{M}_{g_l}(k)\mathcal{M}_{g_l}(k)\right)_{0_{\mathbb{Z}^2},p}}\right\vert.
\end{array}$$
Combinig this with $(4)$ leads to: 
$$\begin{array}{rcl}
&&\displaystyle{\sum_{l\in \mathbb{N}_L}\sum_{n\in \mathbb{Z}^2}\sum_{m\in \mathbb{N}_M^2}\left\vert \langle E_{\frac{m}{M}}T_{nN}g_l,h\rangle \right\vert^2}\\
&\geqslant& M^2\displaystyle{\sum_{k\in \mathbb{Z}^2}\vert f(k)\vert^2\left[ (\sum_{l\in \mathbb{N}_L}\mathcal{M}_{g_l}(k)\mathcal{M}_{g_l}^t(k))_{0_{\mathbb{Z}^2},0_{\mathbb{Z}^2}}-\sum_{0_{\mathbb{Z}^2}\neq p\in \mathbb{Z}^2}\left\vert (\sum_{l\in \mathbb{N}_L}\mathcal{M}_{g_l}(k)\mathcal{M}_{g_l}^t(k)\,)_{0_{\mathbb{Z}^2},p}\right\vert \right]}.
\end{array}$$
Combining this with the hypothesis and the $N$-periodicity of the terms leads to:
$$\sum_{l\in \mathbb{N}_L}\sum_{n\in \mathbb{Z}^2}\sum_{m\in \mathbb{N}_M^2}\left\vert \langle E_{\frac{m}{M}}T_{nN}g_l,h\rangle \right\vert^2\geqslant A\|f\|^2,$$
this for all $f\in \ell_0(\mathbb{Z}\times\mathbb{Z},\mathbb{H})$ wich is dense in $\ell^2(\mathbb{Z}\times \mathbb{Z},\mathbb{H})$. Hence, $\mathcal{G}(g,L,M,N)$ is a Gabor frame with frame bounds $A$ and $B$.
\end{proof}	
The following result is a direct consequence of the theorem above.
\begin{corollary}
If $g:=\{g_l\}_{l\in \mathbb{N}_L}\subset \ell_0(\mathbb{Z}\times \mathbb{Z},\mathbb{R})$, then $\mathcal{G}(g,L,M,N)$ is a Bessel Gabor system in $\ell^2(\mathbb{Z}\times \mathbb{Z},\mathbb{H})$, for arbitrary $L,M,N\in \mathbb{N}$.\\
\end{corollary}
The following proposition presents a necessary and sufficient condition, in a particular case, for a quaternionic Gabor system to be a Gabor frame and  also presents the explicit expression of its frame operator and its inverse.
\begin{proposition}
Let $g:=\{g_l\}_{l\in \mathbb{N}_L}\subset \ell_0(\mathbb{Z}\times \mathbb{Z},\mathbb{R})$ such that for all $l\in \mathbb{N}_L$, $\vert supp(g_l)\vert< M$. Then, the following statements are equivalent:
\begin{enumerate}
\item $\mathcal{G}(g,L,M,N)$ is a frame for $\ell^2(\mathbb{Z}\times \mathbb{Z},\mathbb{H})$ with frame bounds $A\leq B$.
\item For all $k\in \mathbb{Z}^2$,
$$\frac{A}{M^2}\leq \sum_{l\in \mathbb{N}_L}\left(\mathcal{M}_{g_l}(k)\mathcal{M}_{g_l}^t(k)\right)_{0_{\mathbb{Z}^2},0_{\mathbb{Z}^2}}\leq \frac{B}{M^2}.$$
\item For all $k\in \mathbb{N}_N^2$,
$$\frac{A}{M^2}\leq \sum_{l\in \mathbb{N}_L}\left(\mathcal{M}_{g_l}(k)\mathcal{M}_{g_l}^t(k)\right)_{0_{\mathbb{Z}^2},0_{\mathbb{Z}^2}}\leq \frac{B}{M^2}.$$
\end{enumerate}
If one of these statements holds, the frame operator of \( \mathcal{G}(g,L,M,N) \) is explicitly defined for all \( h \in \ell^2(\mathbb{Z} \times \mathbb{Z}, \mathbb{H}) \) by:
$$
Sh = M^2  (\sum_{l \in N_L}M_{g_l}( \cdot ) M_{g_l}^t( \cdot ))_{0_{\mathbb{Z}^2},0_{\mathbb{Z}^2}} h.$$
And: 
$$
S^{-1}h:=\displaystyle{\frac{1}{M^2(\displaystyle{\sum_{l \in N_L}M_{g_l}( \cdot ) M_{g_l}^t( \cdot ))_{0_{\mathbb{Z}^2},0_{\mathbb{Z}^2}}}}}h.
$$
\end{proposition}
\begin{proof}
Note that, in this proof, we will use the notation from Lemma \ref{lem3}. 
By $N\mathbb{Z}^2$-periodicity of $\mathcal{M}_{g_l}(.)\mathcal{M}_{g_l}^t(.)$ (Lemma \ref{lem5}), we have $2.$ and $3.$ are equivalent.   By Proposition \ref{prop1}, it only remains to show that $2.$ implies $1.$. Since for all $l\in \mathbb{N}_L$, $\vert supp(g_l)\vert< M$, then $F_2(h)=0$ for all $h\in \ell_0(\mathbb{Z}\times \mathbb{Z},\mathbb{H})$. Hence, for all $h\in \ell_0(\mathbb{Z}\times \mathbb{Z},\mathbb{H})$, $\displaystyle{\sum_{l\in \mathbb{N}_L}\sum_{n\in \mathbb{Z}^2}\sum_{m\in \mathbb{N}_M^2} \left\vert \langle E_{\frac{m}{M}}T_{nN}g_l,h\rangle \right\vert^2}=F_1(h)$. By, hypothesis, we have that for all $k\in \mathbb{Z}^2$,
$$\frac{A}{M^2}\leq \sum_{l\in \mathbb{N}_L}\left(\mathcal{M}_{g_l}(k)\mathcal{M}_{g_l}^t(k)\right)_{0_{\mathbb{Z}^2},0_{\mathbb{Z}^2}}\leq \frac{B}{M^2}.$$
Then, $A\| h\|^2\leq F_1(h)\leq B\|h\|^2$. Hence, for all $h\in \ell_0(\mathbb{Z}\times \mathbb{Z},\mathbb{H})$, $$A\|h\|^2\leq \displaystyle{\sum_{l\in \mathbb{N}_L}\sum_{n\in \mathbb{Z}^2}\sum_{m\in \mathbb{N}_M^2} \left\vert \langle E_{\frac{m}{M}}T_{nN}g_l,h\rangle \right\vert^2}\leq B\|h\|^2.$$
By density of $\ell_0(\mathbb{Z}\times \mathbb{Z},\mathbb{H})$, we deduce that $\mathcal{G}(g,L,M,N)$ is a frame with frame bounds $A\leq B$. Assume, now, that one of these statements holds. Then for all \( h \in \ell_0(\mathbb{Z}\times \mathbb{Z},\mathbb{H}) \), we have:
$$\begin{array}{rcl}
\langle Sh, h \rangle = F_1(h)& = &M^2\displaystyle{\sum_{k \in \mathbb{Z}^2} ( \sum_{l \in N_L} \mathcal{M}_{g_l}(k) \mathcal{M}_{g_l}^t(k) )_{0_{\mathbb{Z}^2},0_{\mathbb{Z}^2}} |h(k)|^2}\\
&=&\displaystyle{\sum_{k\in \mathbb{Z}^2}  M^2( \sum_{l \in \mathbb{N}_L}\mathcal{M}_{g_l}(j) \mathcal{M}_{g_l}^t(j) )_{0_{\mathbb{Z}^2},0_{\mathbb{Z}^2}} \overline{h(k)} h(k)}\\
&=&\left\langle M^2 \displaystyle{ ( \sum_{l \in \mathbb{N}_L}\mathcal{M}_{g_l}( \cdot ) \mathcal{M}_{g_l}^t( \cdot ) )_{0_{\mathbb{Z}^2},0_{\mathbb{Z}^2}} h}, h \right\rangle.
\end{array}$$

By density of \( \ell_0(\mathbb{Z}\times \mathbb{Z},\mathbb{H}) \) in \( \ell_2(\mathbb{Z}\times \mathbb{Z},\mathbb{H}) \) and since \( h \mapsto M^2 ( \displaystyle{\sum_{l \in N_L}M_{g_l}( \cdot ) M_{g_l}^t( \cdot ) })_{0_{\mathbb{Z}^2},0_{\mathbb{Z}^2}} h \) is right $\mathbb{H}$-linear, bounded and self-adjoint, we have for all $h\in \ell^2(\mathbb{Z}\times \mathbb{Z},\mathbb{H})$:
$$
Sh = M^2  (\sum_{l \in N_L}M_{g_l}( \cdot ) M_{g_l}^t( \cdot ))_{0_{\mathbb{Z}^2},0_{\mathbb{Z}^2}} h.$$
And: 
$$
S^{-1}h:=\displaystyle{\frac{1}{M^2(\displaystyle{\sum_{l \in N_L}M_{g_l}( \cdot ) M_{g_l}^t( \cdot ))_{0_{\mathbb{Z}^2},0_{\mathbb{Z}^2}}}}}h.
$$
\end{proof}

The following theorem characterizes Parseval Gabor frames with real-valued windows by the matrix-valued functions $\mathcal{M}_{.}$.	
\begin{theorem}\label{thm5}
Let $g:=\{g_l\}_{l\in \mathbb{N}_L}\subset \ell^2(\mathbb{Z}\times \mathbb{Z},\mathbb{R})$. Then the following statements are equivalent:
\begin{enumerate}
\item $\mathcal{G}(g,L,M,N)$ is a Parseval Gabor frame for $\ell^2(\mathbb{Z}\times \mathbb{Z},\mathbb{H})$.
\item For all $k\in \mathbb{N}_N^2$, $p\in \mathbb{Z}^2$, $\left(\mathcal{M}_{g_l}(k)\mathcal{M}_{g_l}^t(k)\right)_{0_{\mathbb{Z}^2},p}=\displaystyle{\frac{1}{M^2}}\chi_{\{0_{\mathbb{Z}^2}\}}(p)$.
\end{enumerate}
\end{theorem}	
\begin{proof}
Note that in this proof, we will use the notation from Lemma \ref{lem3}. Assume $2.$. Then, by Lemma \ref{lem3}, $F_1(h)=\|h\|^2$ and $F_2(h)=0$ for all $h\in \ell_0(\mathbb{Z}\times \mathbb{Z},\mathbb{H})$. Hence, by density of $\ell_0(\mathbb{Z}\times \mathbb{Z},\mathbb{H})$ in $\ell^2(\mathbb{Z}\times \mathbb{Z},\mathbb{H})$, we deduce that $\mathcal{G}(g,L,M,N)$ is a Parseval frame for $\ell^2(\mathbb{Z}\times \mathbb{Z},\mathbb{H})$. Conversely, assume that $\mathcal{G}(g,L,M,N)$ is a Parseval Gabor frame, then by Proposition \ref{prop1}, for all $k\in \mathbb{N}_N^2$, $\left(\mathcal{M}_{g_l}(k)\mathcal{M}_{g_l}(k)\right)_{0_{\mathbb{Z}^2},0_{\mathbb{Z}^2}}=\displaystyle{\frac{1}{M^2}}$. Thus, $F_1(h)=\|h\|^2$ for all $h\in \ell_0(\mathbb{Z}\times \mathbb{Z},\mathbb{H})$. Then, $F_2(h)=0$ for all $h\in \ell_0(\mathbb{Z}\times \mathbb{Z},\mathbb{H})$. Fix $k'\in \mathbb{N}_N^2$ and $p'\in \mathbb{Z}^2-\{0_{\mathbb{Z}^2}\}$ and set $h\in \ell_0(\mathbb{Z}\times \mathbb{Z},\mathbb{H})$ defined by $h(k')=h(k'+p'M)=1$ and $h(k)=0$ elsewhere. Then, $F_2(h)=M^2\left(\displaystyle{\sum_{l\in \mathbb{N}_L}\mathcal{M}_{g_l}(k')\mathcal{M}_{g_l}^t(k')}\right)_{0_{\mathbb{Z}^2},p'}=0$. Hence, by arbitrariness of $k'$ and $p'$, we deduce that for all $k\in \mathbb{N}_N^2$ and $p\in \mathbb{Z}^2-\{0_{\mathbb{Z}^2}\}$, $\left(\displaystyle{\sum_{l\in \mathbb{N}_L}\mathcal{M}_{g_l}(k)\mathcal{M}_{g_l}^t(k)}\right)_{0_{\mathbb{Z}^2},p}=0$. 
\end{proof}	
\begin{remark}
If $\vert supp(g_l)\vert< M$, then for all $k\in \mathbb{Z}^2$ and $p\in \mathbb{Z}^2-\{0_{\mathbb{Z}^2}\}$, we have: $\left(\displaystyle{\mathcal{M}_{g_l}(k)\mathcal{M}_{g_l}^t(k)}\right)_{0_{\mathbb{Z}^2},p}=0$.\\
\end{remark}
Here, we present a theoretical constructive example of Parseval Gabor frames in the case where the number of windows is a perfect square. 

\begin{example}
Let $L,M,N\in \mathbb{N}$. Write $L=K^2$ and suppose that $L$ satisfy the necessary condition seen in Proposition \ref{prop}, i.e., $N^2\leq LM^2$. Here, we are concerned with the strict inequality. The case of equality will be studied in Section 3.., i.e., we suppose that $N^2< LM^2$. Then, $N< KM$. Let $R\in \mathbb{N}$ be the maximal integer satisfying $RM\leq N$. For all $r\in \mathbb{N}_R$, define $I_r$ as the set of the $(r+1)$-th $M$ elements of $\mathbb{N}_N$ and $I_R$ as the set of the rest elements of $\mathbb{N}_N$.
We have, then, $(I_r)_{r\in \mathbb{N}_{R+1}}$ is a partition of $\mathbb{N}_N$ and $card(I_r)=M$ for all $r\in \mathbb{N}_R$ and $card(I_R)< M$. Define, now, for all $r,r'\in \mathbb{N}_{R+1}$, $A_{(r,r')}:=I_r\times I_{r'}$ and $g_{(r,r')}:=\displaystyle{\frac{1}{M}\chi_{A_{(r,r')}}}$, and if $R+1< K$, set $g_{(r,r')}:=0$ for all $R+1\leq r,r'\leq K-1$. Observe that $\vert supp(g_{(r,r')})\vert<M$, then, for all $k\in \mathbb{N}_N^2$ and $p\in \mathbb{Z}^2-\{0_{\mathbb{Z}^2}\}$,  $\left(\displaystyle{\mathcal{M}_{g_{(r,r')}}(k)\mathcal{M}_{g_{(r,r')}}^t(k)}\right)_{0_{\mathbb{Z}^2},p}=0$. And by a simple computation, we obtain that, for all $k\in \mathbb{N}_N^2$, $\left(\displaystyle{\mathcal{M}_{g_{(r,r')}}(k)\mathcal{M}_{g_{(r,r')}}^t(k)}\right)_{0_{\mathbb{Z}^2},0_{\mathbb{Z}^2}}=\displaystyle{\frac{1}{M^2}}$. Hence, Theorem \ref{thm5} shows that $\mathcal{G}(g,L,M,N)$ is a Parseval frame for $\ell^2(\mathbb{Z}\times \mathbb{Z},\mathbb{H})$, where $g:=\{g_{(r,r')}\}_{r,r'\in \mathbb{N}_K}$.
\end{example}

\section{Discrete quaternionic Gabor orthonormal bases}
In this section, we investigate Gabor orthonormal bases for the space \( \ell^2(\mathbb{Z} \times \mathbb{Z}, \mathbb{H}) \). First, we characterize the quaternionic Gabor systems that form orthonormal bases for \( \ell^2(\mathbb{Z} \times \mathbb{Z}, \mathbb{H}) \) within the class of Parseval quaternionic Gabor frames, and we give an interesting property regarding the number of windows in a Gabor orthonormal basis. Then, we characterize such systems using matrix-valued functions \( \mathcal{M}_{.} \), and finally, we analyze the admissibility of these systems in \( \ell^2(\mathbb{Z} \times \mathbb{Z}, \mathbb{H}) \).
\begin{theorem}\label{thm2}
Let $g:=\{g_l\}_{l\in \mathbb{N}_L}\subset \ell^2(\mathbb{Z}\times \mathbb{Z},\mathbb{H})$. Assume that $\mathcal{G}(g,L,M,N)$ is a Parseval farame. Then, the following statements are equivalent:
\begin{enumerate}
\item $\mathcal{G}(g,L,M,N)$ is a an orthonormal basis for $\ell^2(\mathbb{Z}\times \mathbb{Z},\mathbb{H})$.
\item $N^2=LM^2.$\\
\end{enumerate}
\end{theorem}

\begin{proof}
Assume that $\mathcal{G}(g,L,M,N)$ is a Parseval farame. If $\mathcal{G}(g,L,M,N)$ is an orthonormal basis, then $\|g_l\|=\|E_{\frac{m}{M}}T_{nN}g_l\|= 1$ for all $l\in \mathbb{N}_L$, where $n\in \mathbb{Z}^2$ and $m\in \mathbb{N}_M^2$ are arbitrary. Then, by $(1)$, we obtain that $L=\displaystyle{\frac{N^2}{M^2}}$. Conversely, if $N^2=LM^2$, then by $(1)$, we obtain that $\displaystyle{\sum_{l\in \mathbb{N}_L}\|g_l\|^2=L}$. Combining this with the fact that $\|g_l\|\leq 1$ for all $l\in \mathbb{N}_L$, we conclude that $\|g_l\|=1$ for all $l\in \mathbb{N}_L$. Hence, $\mathcal{G}(g,L,M,N)$ is an orthonormal basis for $\ell^2(\mathbb{Z}\times \mathbb{Z},\mathbb{H})$.\\
\end{proof}

The following lemma shows that the number of windows in a  Gabor orthonormal basis must necessarily be a perfect square.
\begin{lemma}\label{lm}
Let $g:=\{g_l\}_{l\in \mathbb{N}_L}\subset \ell^2(\mathbb{Z}\times \mathbb{Z},\mathbb{H})$ such that $\mathcal{G}(g,L,M,N)$ is an orthonormal basis for $\ell^2(\mathbb{Z}\times \mathbb{Z},\mathbb{H})$. Then, $L$ is a perfect square.
\end{lemma}
\begin{proof}
Let $p,q\in \mathbb{N}$ such that $p\wedge q=1$ and $\displaystyle{\frac{N}{M}=\frac{p}{q}}$. by Theorem \ref{thm2}, we have that $N^2=LM^2$, then $p^2=Lq^2$. Since $p\wedge q=1$, then $p^2\wedge q^2=1$. Since $p^2$ devides $Lq^2$, then $p^2$ devides $L$. Let $k\in \mathbb{N}$ such that $L=p^2k$, then $p^2=p^2kq^2$, then $kq^2=1$, thus $k=1$. Hence, $L=p^2$.
\end{proof}

The following theorem characterizes  Gabor systems with real-valued windows that are orthonormal bases for $\ell^2(\mathbb{Z}\times \mathbb{Z},\mathbb{H})$ By the matrix-valued functions $\mathcal{M}_{.}$.
\begin{theorem}\label{thm6}
Let $g:=\{g_l\}_{l\in \mathbb{N}_L}\subset \ell^2(\mathbb{Z}\times \mathbb{Z},\mathbb{R})$. Then the following statements are equivalent:
\begin{enumerate}
\item $\mathcal{G}(g,L,M,N)$ is a an orthonormal basis  for $\ell^2(\mathbb{Z}\times \mathbb{Z},\mathbb{H})$.
\item \begin{enumerate}
\item $N^2=LM^2$,
\item For all $k\in \mathbb{N}_N^2$, $p\in \mathbb{Z}^2$, $\left(\mathcal{M}_{g_l}(k)\mathcal{M}_{g_l}^t(k)\right)_{0_{\mathbb{Z}^2},p}=\displaystyle{\frac{1}{M^2}}\chi_{\{0_{\mathbb{Z}^2}\}}(p)$.
\end{enumerate}
\end{enumerate}
\end{theorem}	
\begin{proof}
By Theorem \ref{thm5} and Theorem \ref{thm2} together.
\end{proof}	
The following theorem characterizes the existence of orthonormal Gabor bases for $\ell^2(\mathbb{Z}\times \mathbb{Z},\mathbb{H})$  in terms of the parameters \( L \), \( M \), and \( N \).
\begin{theorem}\label{thm8}
Let $M,N\in \mathbb{N}$.  The following statements are equivalent:
\begin{enumerate}
\item There exists $g:=\{g_l\}_{l\in \mathbb{N}_L}\subset \ell^2(\mathbb{Z}\times \mathbb{Z},\mathbb{H})$ such that $\mathcal{G}(g,L,M,N)$ is an orthonormal basis for $\ell^2(\mathbb{Z}\times \mathbb{Z},\mathbb{H})$.
\item $N^2= LM^2$.
\end{enumerate}
\end{theorem}
\begin{proof}
By Theorem \ref{thm2}, we have  $1.$ implies $2.$. Conversely, assume that $N^2= LM^2$. By the proof of  Lemma  \ref{lm}, we can write $L=K^2$, where $K\in \mathbb{N}$. Then, $N=KM$. For all $k\in \mathbb{N}_K$, define $I_k$ as the set of the $(k+1)$-th $M$ elements of $\mathbb{N}_N$. Then, $(I_k)_{k\in \mathbb{N}_K}$ is a partition of $\mathbb{N}_N$ where $card(I_k)=M$ for all $k\in \mathbb{N}_K$. Define, now, for all $k,k'\in \mathbb{N}_K$, $A_{(k,k')}:=I_k\times I_{k'}$ and then $g_{(k,k')}:=\displaystyle{\frac{1}{M}\chi_{A_{(k,k')}}}$. By a simple calculation, we obtain:
$\left(\mathcal{M}_{g_{(k,k')}}(k)\mathcal{M}_{g_{(k,k')}}^t(k)\right)_{0_{\mathbb{Z}^2},p}=0$ for all $k\in \mathbb{N}_N^2$ and $p\in\mathbb{Z}^2-\{0_{\mathbb{Z}^2}\}$ since $\vert supp(g_{(k,k')})\vert =M-1< M$ for all $l\in \mathbb{N}_L$, and $\left(\mathcal{M}_{g_{(k,k')}}(k)\mathcal{M}_{g_{(k,k')}}^t(k)\right)_{0_{\mathbb{Z}^2},0_{\mathbb{Z}^2}}=\displaystyle{\frac{1}{M^2}}$. Then, Theorem \ref{thm6} shows that $\mathcal{G}(g,L,M,N)$ is a Gabor orthonormal basis for $\ell^2(\mathbb{Z}\times \mathbb{Z},\mathbb{H})$, where $g:=\{g_{(r,r')}\}_{r,r'\in \mathbb{N}_K}$.
\end{proof}	

\begin{example}\hspace{1cm}
\begin{itemize}
\item For $N=10$ and $M=4$, there does not exist a Gabor orthonormal basis for $\ell^2(\mathbb{Z}\times \mathbb{Z},\mathbb{H})$.
\item Let $N=10$ and $M=5$. All associated Gabor orthonormal bases for $\ell^2(\mathbb{Z}\times \mathbb{Z},\mathbb{H})$ must be constructed with $4$ windows. Here, we present an example of one such basis. Define:
$A_0:=\{0,1,2,3,4\}\times \{0,1,2,3,4\}$, $A_1:=\{0,1,2,3,4\}\times \{5,6,7,8,9\}$, $A_2:=\{5,6,7,8,9\}\times \{0,1,2,3,4\}$ and $A_3:=\{5,6,7,8,9\}\times \{5,6,7,8,9\}$. Hence, $\mathcal{G}(g,4,5,10)$ is an orthonormal basis for $\ell^2(\mathbb{Z}\times \mathbb{Z},\mathbb{H})$.
\end{itemize}
\end{example}

\section{Duality of discrete quaternionic Gabor frames}
In frame theory, there is a strong relationship between Parseval frames and duality; in fact, a frame is Parseval if and only if it is dual to itself. We will, therefore, be inspired by Theorem \ref{thm5} to characterize the duality of two quaternionic Gabor systems with real-valued windows.

Denote by  $\ell_M(\mathbb{Z}\times \mathbb{Z},\mathbb{H})$, the set of $M$-periodic sequences on $\mathbb{Z}\times \mathbb{Z}$. $\ell_M(\mathbb{Z}\times \mathbb{Z},\mathbb{H})$ equipped with the inner product defined for $f,g\in \ell_M(\mathbb{Z}\times \mathbb{Z},\mathbb{H})$ by $\langle f,g \rangle:=\displaystyle{\sum_{k\in \mathbb{N}_M^2}\overline{f(k)}g(k)}$ is a a finite-dimensional  right quaternionic Hilbert space with repect to the right scalar multiplication. The following lemma presents an orthonormal basis for this space.
\begin{lemma}\label{lem2}\vspace{0.5cm}
$\left\{ \; (k_1,k_2)\mapsto\displaystyle{\frac{1}{M}e^{2\pi i\frac{m_1}{M}k_1}e^{2\pi j\frac{m_2}{M}k_2}}\right\}_{m_1,m_2\in \mathbb{N}_M}$ is an orthonormal basis for $\ell_M(\mathbb{Z}\times \mathbb{Z},\mathbb{H})$.\\
\end{lemma}
We will need the following useful lemma.
\begin{lemma}\label{lem4}
Let $\{g_l\}_{l\in \mathbb{N}_L},\{h_l\}_{l\in \mathbb{N}_L}\subset \ell^2(\mathbb{Z}\times \mathbb{Z},\mathbb{R})$. Then, we have for all  $f,\phi\in \ell_0(\mathbb{Z}\times \mathbb{Z},\mathbb{H})$:
$$
\begin{array}{rcl}
&&\displaystyle{\sum_{l\in \mathbb{N}_L}\sum_{n\in \mathbb{Z}^2}\sum_{m\in \mathbb{N}_M^2}\langle f,E_{\frac{m}{M}}T_{nN}g\rangle\langle E_{\frac{m}{M}}T_{nN}h,\phi\rangle}\\
&=&M^2\displaystyle{\sum_{k\in \mathbb{Z}^2}\sum_{p\in \mathbb{Z}^2} \left(\sum_{l\in \mathbb{N}_L}\mathcal{M}_{g_l}(k)\mathcal{M}_{h_l}^t(k)\right)_{0_{\mathbb{Z}^2},p} \overline{f(k)}\phi(k+pM).}
\end{array}$$
\end{lemma}	
\begin{proof}
Note that in this proof, we denote \( a = (a_1, a_2) \) for every \( a \in \mathbb{Z}^2 \) and, without any possible confusion, the inner product in \( \ell^2(\mathbb{Z}\times \mathbb{Z},\mathbb{H}) \) and that in \( \ell_M(\mathbb{Z}\times \mathbb{Z},\mathbb{H}) \) will be denoted by the same notation \( \langle \cdot, \cdot \rangle \). Let $f,\phi\in \ell_0(\mathbb{Z}\times \mathbb{Z},\mathbb{H})$, we have:
$$
\begin{array}{rcl}
&&\displaystyle{\sum_{l\in \mathbb{N}_L}\sum_{n\in \mathbb{Z}^2}\sum_{m\in \mathbb{N}_M^2}\langle f,E_{\frac{m}{M}}T_{nN}g_l\rangle\langle E_{\frac{m}{M}}T_{nN}h_l,\phi\rangle}\\
&=&\displaystyle{\sum_{l\in \mathbb{N}_L}\sum_{n\in \mathbb{Z}^2}\sum_{m\in \mathbb{N}_M^2}\left(\sum_{k\in \mathbb{Z}^2}\overline{f(k)}e^{2\pi i \frac{m_1}{M}k_1}g_l(k-nN)e^{2\pi j \frac{m_2}{M}k_2}\right)}\\
&\times& \left(\displaystyle{\sum_{k\in\mathbb{Z}^2}e^{-2\pi j\frac{m_2}{M}k_2}\overline{h(k-nN)}e^{-2\pi i \frac{m_1}{M}k_1}\phi(k)}\right).\\
&=&\displaystyle{\sum_{l\in \mathbb{N}_L}\sum_{n\in \mathbb{Z}^2}\sum_{m\in \mathbb{N}_M^2}\left(\sum_{k\in \mathbb{N}_M^2}\sum_{p\in \mathbb{Z}^2}\overline{f(k+pM)}e^{2\pi i \frac{m_1}{M}k_1}g_l(k+pM-nN)e^{2\pi j \frac{m_2}{M}k_2}\right)}\\
&\times& \left(\displaystyle{\sum_{k\in\mathbb{N}_M^2}\sum_{p\in \mathbb{Z}^2} e^{-2\pi j\frac{m_2}{M}k_2}\overline{h(k+pM-nN)}e^{-2\pi i \frac{m_1}{M}k_1}\phi(k+pM)}\right).\\
&=&\displaystyle{\sum_{l\in \mathbb{N}_L}\sum_{n\in \mathbb{Z}^2}\sum_{m\in \mathbb{N}_M^2}\left(\sum_{k\in \mathbb{N}_M^2}\sum_{p\in \mathbb{Z}^2}\overline{f(k+pM)}g_l(k+pM-nN) e^{2\pi i \frac{m_1}{M}k_1}e^{2\pi j \frac{m_2}{M}k_2}\right)}\\
&\times& \left(\displaystyle{\sum_{k\in\mathbb{N}_M^2}\sum_{p\in \mathbb{Z}^2} e^{-2\pi j\frac{m_2}{M}k_2}e^{-2\pi i \frac{m_1}{M}k_1}h(k+pM-nN)\phi(k+pM)}\right).\\
&=&\displaystyle{\sum_{l\in \mathbb{N}_L}\sum_{n\in \mathbb{Z}^2}\sum_{m\in \mathbb{N}_M^2}\left(\sum_{k\in \mathbb{N}_M}^2 \overline{F_n(k)}e^{2\pi i \frac{m_1}{M}k_1}e^{2\pi j \frac{m_2}{M}k_2}\right)}\\
&\times& \left(\displaystyle{\sum_{k\in \mathbb{N}_M^2}e^{-2\pi i \frac{m_1}{M}k_1}e^{-2\pi j\frac{m_2}{M}k_2}H_n(k)}\right)\\
&=&\displaystyle{\sum_{l\in \mathbb{N}_L}\sum_{n\in \mathbb{Z}^2}\sum_{m\in \mathbb{N}_M^2}\langle F_n(k),e^{2\pi i \frac{m_1}{M}.}e^{2\pi j \frac{m_2}{M}.}\rangle \langle e^{2\pi i \frac{m_1}{M}.}e^{2\pi j \frac{m_2}{M}.}\rangle}\\
&=&M^2\displaystyle{\sum_{l\in \mathbb{N}_L}\sum_{n\in \mathbb{Z}^2}\langle F_n(k),H_n(k)\rangle}\\
\end{array}$$

$$\begin{array}{rcl}
&=&M^2\displaystyle{\sum_{l\in \mathbb{N}_L}\sum_{n\in \mathbb{Z}^2}\sum_{k\in\mathbb{N}_M}\overline{F_n(k)}H_n(k)}\\
&=&M^2\displaystyle{\sum_{l\in \mathbb{N}_L}\sum_{n\in \mathbb{Z}^2}\sum_{k\in \mathbb{N}_M^2}\sum_{p\in \mathbb{Z}^2}\overline{f(k+pM)}g(k+pM-nN)H_n(k)}\\
&=&M^2\displaystyle{\sum_{l\in \mathbb{N}_L}\sum_{n\in \mathbb{Z}^2}\sum_{k\in \mathbb{N}_M^2}\sum_{p\in \mathbb{Z}^2}\overline{f(k+pM)}g(k+pM-nN)H_n(k+pM)}\\
&=&M^2\displaystyle{\sum_{l\in \mathbb{N}_L}\sum_{n\in \mathbb{Z}^2}\sum_{k\in \mathbb{Z}^2}\overline{f(k)}g(k-nN)\sum_{p\in \mathbb{Z}^2}h(k+pM-nN)\phi(k+pM)}\\
&=&M^2\displaystyle{\sum_{l\in \mathbb{N}_L}\sum_{k\in \mathbb{Z}^2}\sum_{p\in \mathbb{Z}^2}\sum_{n\in \mathbb{Z}^2}g(k-nN)h(k+pM-nN)\overline{f(k)}\phi(k+pM)}\\
&=&M^2\displaystyle{\sum_{l\in \mathbb{N}_L}\sum_{k\in \mathbb{Z}^2}\sum_{p\in \mathbb{Z}^2}\left(\mathcal{M}_{g_l}(k)\mathcal{M}_{h_l}(k)\right)_{0_{\mathbb{Z}^2},p}\overline{f(k)}\phi(k+pM)}\\
&=&M^2\displaystyle{\sum_{k\in \mathbb{Z}^2}\sum_{p\in \mathbb{Z}^2}\left(\sum_{l\in \mathbb{N}_L}\mathcal{M}_{g_l}(k)\mathcal{M}_{h_l}(k)\right)_{0_{\mathbb{Z}^2},p}\overline{f(k)}\phi(k+pM)},\\
\end{array}$$
where $F_n(.):=\displaystyle{\sum_{p\in \mathbb{Z}^2}f(.+pM)g(.+pM-nN)}$ and $H_n(.):=\displaystyle{\sum_{p\in \mathbb{Z}^2}\phi(.+pM)h(.+pM-nN)}$ which are $M$-periodic. The above change of summation is justified by the fact that \( f,\phi\in \ell_0(\mathbb{Z} \times \mathbb{Z}, \mathbb{H}) \), the commutativity of \( g_l \) and \(h_l\) with the other terms is justified by the fact that they take real values, and note that we used Lemma \ref{lem2}.
\end{proof}

We have arrived at the following theorem that characterizes the duality of two Gabor systems in $\ell^2(\mathbb{Z}\times \mathbb{Z},\mathbb{H})$ with real-valued windows.
\begin{theorem}\label{thm9}
Let $g:=\{g_l\}_{l\in \mathbb{N}_L},\; h:=\{h_l\}_{l\in \mathbb{N}_L}\subset \ell^2(\mathbb{Z}\times \mathbb{Z},\mathbb{H})$. Then, the following statements are equivalent:
\begin{enumerate}
\item $\mathcal{G}(g,L,M,N)$ and $\mathcal{G}(h,L,M,N)$ are dual.
\item For all $k\in \mathbb{N}_N^2$ and $p\in \mathbb{Z}^2$, $\left(\displaystyle{\sum_{l\in \mathbb{N}_L}\mathcal{M}_{g_l}(k)\mathcal{M}_{h_l}^t(k)}\right)_{0_{\mathbb{Z}^2},p}=\displaystyle{\frac{1}{M^2}}\chi_{\{0_{\mathbb{Z}^2}\}}(p).$
\end{enumerate}
\end{theorem}
\begin{proof}
By Lemma \ref{lem4}, it is clear that $2.$ implies  $1.$. Conversely, assume that $\mathcal{G}(g,L,M,N)$ and $\mathcal{G}(h,L,M,N)$ are dual. Fix $k'\in \mathbb{N}_N^2$ and set $f,\phi\in \ell_0(\mathbb{Z}\times \mathbb{Z},\mathbb{H})$ defined by  $f=\phi:=\chi_{\{k'\}}$. By Lemma \ref{lem4}, we obtain that $\left(\displaystyle{\sum_{l\in \mathbb{N}_L}\mathcal{M}_{g_l}(k')\mathcal{M}_{h_l}^t(k')}\right)_{0_{\mathbb{Z}^2},0_{\mathbb{Z}^2}}=\displaystyle{\frac{1}{M^2}}$. Then, by Lemma \ref{lem4}, we have for all  $f,\phi\in \ell_0(\mathbb{Z}\times \mathbb{Z},\mathbb{H})$:
$$
\begin{array}{rcl}
&&\langle f,\phi\rangle =\displaystyle{\sum_{l\in \mathbb{N}_L}\sum_{n\in \mathbb{Z}^2}\sum_{m\in \mathbb{N}_M^2}\langle f,E_{\frac{m}{M}}T_{nN}g\rangle\langle E_{\frac{m}{M}}T_{nN}h,\phi\rangle}\\
&=&M^2\displaystyle{\sum_{k\in \mathbb{Z}^2}\sum_{p\in \mathbb{Z}^2} \left(\sum_{l\in \mathbb{N}_L}\mathcal{M}_{g_l}(k)\mathcal{M}_{h_l}^t(k)\right)_{0_{\mathbb{Z}^2},p} \overline{f(k)}\phi(k+pM).}\\
&=&\langle f, \phi\rangle +M^2\displaystyle{\sum_{k\in \mathbb{Z}^2}\sum_{0_{\mathbb{Z}^2}\neq p\in \mathbb{Z}^2} \left(\sum_{l\in \mathbb{N}_L}\mathcal{M}_{g_l}(k)\mathcal{M}_{h_l}^t(k)\right)_{0_{\mathbb{Z}^2},p} \overline{f(k)}\phi(k+pM).}
\end{array}$$
Then, \begin{equation}
\displaystyle{\sum_{k\in \mathbb{Z}^2}\sum_{0_{\mathbb{Z}^2}\neq p\in \mathbb{Z}^2} \left(\sum_{l\in \mathbb{N}_L}\mathcal{M}_{g_l}(k)\mathcal{M}_{h_l}^t(k)\right)_{0_{\mathbb{Z}^2},p} \overline{f(k)}\phi(k+pM)}=0,
\end{equation} for all $f,\phi\in \ell_0(\mathbb{Z}\times \mathbb{Z},\mathbb{H})$.
Fix $k'\in \mathbb{N}_N^2$ and $p'\in \mathbb{Z}^2-\{0_{\mathbb{Z}^2}\}$ and set $f,\phi \in \ell_0(\mathbb{Z}\times\mathbb{Z},\mathbb{H})$ defined by $f:=\chi_{\{k'\}}$ and $\phi:=\chi_{\{k'+p'M\}}$. Then, by $(5)$, we obtain: $$\left(\displaystyle{\sum_{l\in \mathbb{N}_L}\mathcal{M}_{g_l}(k')\mathcal{M}_{h_l}^t(k')}\right)_{0_{\mathbb{Z}^2},p'}=0.$$
Hence, by arbitrariness of $k'$ and $p'$, we deduce that for all $k\in \mathbb{N}_N^2$ and $p\in \mathbb{Z}^2-\{0_{\mathbb{Z}^2}\}$, $$\left(\displaystyle{\sum_{l\in \mathbb{N}_L}\mathcal{M}_{g_l}(k)\mathcal{M}_{h_l}^t(k)}\right)_{0_{\mathbb{Z}^2},p}=0.$$
\end{proof}
\section{Stability of quaternionic Gabor frames}
In this section, we provide the stability of quetrnionic Gabor frames with real-valued windows under some perturbations. We will need the following lemma from frame thoery in quaternionic Hilbert spaces:
\begin{lemma}\label{lemma}
Let $\mathcal{H}$ be a right quaternionic Hilbert space and $\{u_i\}_{i\in I}$ be a frame for $\mathcal{H}$ with frame bounds $A\leq B$. Let $\{v_i\}_{i\in I}\subset \mathcal{H}$ and $0<R< A$ such that for all $u\in \mathcal{H}$, we have:
$$\sum_{i\in I}\left\vert \langle u_i-v_i,u\rangle \right\vert^2\leq R\|u\|^2.$$
Then, $\{v_i\}_{i\in I}$ is a frame for $\mathcal{H}$ with frame bounds: 
$$A\left(1-\sqrt{\frac{R}{A}}\right)^2,\;\; B\left(1+\sqrt{\frac{R}{B}}\right)^2.$$
\end{lemma}
\begin{proof}
It is clear that $\{u_i-v_i\}_{i\in I}$ is  a Bessel sequence and its transform operator $T^*$ satisfy $\|T^*\|\leq \sqrt{R}$. Then, $\|T\|\leq \sqrt{R}$. Thus, for all $\{q_i\}_{i\in I}\subset \ell^2(\mathbb{H})$, we have:
$$\left\| \sum_{i\in I}(u_i-v_i)q_i\right\|\leq \sqrt{R}\left(\sum_{i\in I}\vert q_i\vert^2\right)^{\frac{1}{2}}.$$
Setting $\lambda=0$ and $\mu=\sqrt{R}$ in Theorem $4.1$ \cite{11}, we deduce that $\{v_i\}_{i\in I}$ is a frame for $\mathcal{H}$ with frame bounds: 
$$A\left(1-\sqrt{\frac{R}{A}}\right)^2,\;\; B\left(1+\sqrt{\frac{R}{B}}\right)^2.$$
\end{proof}

The following theorem studies the stability of Gabor frames with real-valued windows by using matrix-valued functions $\mathcal{M}_{.}$.
\begin{theorem}
Let $g:=\{g_l\}_{l\in \mathbb{N}_L},\, h:=\{h_l\}_{l\in \mathbb{N}_L}\subset \ell^2(\mathbb{Z}\times \mathbb{Z},\mathbb{R})$ such that $\mathcal{G}(g,L,M,N)$ is a frame for $\ell^2(\mathbb{Z}\times \mathbb{Z},\mathbb{H})$ with frame bounds $A\leq B$ and suppose that:
$$R:=M^2\max_{k\in \mathbb{N}_N^2}\sum_{p\in \mathbb{Z}^2}\left\vert \left(\sum_{l\in \mathbb{N}_L}\mathcal{M}_{g_l-h_l}(k)\mathcal{M}_{g_l-h_l}^t(k)\right)_{0_{\mathbb{Z}^2},p}\right\vert < A.$$
Then, $\mathcal{G}(h,L,M,N)$ is a frame for $\ell^2(\mathbb{Z}\times \mathbb{Z},\mathbb{H})$ with frame bounds: 
$$A\left(1-\sqrt{\frac{R}{A}}\right)^2,\;\; B\left(1+\sqrt{\frac{R}{B}}\right)^2.$$
\end{theorem}
\begin{proof}
Applying Theorem \ref{thmd} on $g-h:=\{g_l-h_l\}_{l\in \mathbb{N}_L}$, we deduce that $\mathcal{G}(g-h,L,M,N)$ is a Bessel system with Bessel bound $R$. Hence, for all $f\in \ell^2(\mathbb{Z}\times \mathbb{Z},\mathbb{H})$, we have: 
$$\displaystyle{\sum_{l\in \mathbb{N}_L}\sum_{n\in \mathbb{Z}^2}\sum_{\in \mathbb{N}_M^2}\left\vert \langle E_{\frac{m}{M}}T_{nN}(g_l-h_l),f\rangle\right\vert^2}\leq R\|f\|^2.$$
Since $$\begin{array}{rcl}
E_{\frac{m}{M}}T_{nN}(g_l-h_l)(k)&=&e^{2\pi i\frac{m_1}{M}k_1}(g_l(k)-h_l(k))e^{2\pi j \frac{m_2}{M}k_2}\\
&=&e^{2\pi i\frac{m_1}{M}k_1}g_l(k)e^{2\pi j \frac{m_2}{M}k_2}-e^{2\pi i\frac{m_1}{M}k_1}h_l(k)e^{2\pi j \frac{m_2}{M}k_2}\\
&=&E_{\frac{m}{M}}T_{nN}g_l-E_{\frac{m}{M}}T_{nN}h_l.
\end{array}$$
Then, for all $f\in \ell^2(\mathbb{Z}\times \mathbb{Z},\mathbb{H})$, we have:
$$\displaystyle{\sum_{l\in \mathbb{N}_L}\sum_{n\in \mathbb{Z}^2}\sum_{\in \mathbb{N}_M^2}\left\vert \langle E_{\frac{m}{M}}T_{nN}g_l-E_{\frac{m}{M}}T_{nN}h_l,f\rangle\right\vert^2}\leq R\|f\|^2.$$
Hence, Lemma  \ref{lemma} completes the proof.
\end{proof}
\medskip

\section*{Acknowledgments}
	It is with great pleasure that I thank the referee for their careful reading of the paper and for several valuable suggestions.
	
	\section*{Ethics declarations}
	
	\subsection*{Availablity of data and materials}
	Not applicable.
	\subsection*{Conflict of interest}
	The author declares that hes has no competing interests.
	\subsection*{Fundings}
	Not applicable.

\end{document}